\documentclass[11pt]{amsart}
\usepackage[latin1]{inputenc}
\usepackage[english]{babel}
\usepackage{amssymb,amsfonts, amsmath, color}
\usepackage[margin=2.7cm]{geometry}
\usepackage{xcolor}
\usepackage{graphicx, color, enumerate}
\usepackage[active]{srcltx}
\usepackage{tikz}
\usepackage{pgf}
\usepackage{etex}
\usepackage{verbatim}
\usepackage{tikz-3dplot}
\usepackage{pgfkeys}
\usepackage{amsmath}
\usepackage{amsfonts}
\usepackage{amssymb}
\usepackage{amsthm}
\usepackage{algpseudocode}
\usepackage{float}
\usepackage{xcolor}
\usepackage{mathrsfs}
\usepackage{import}
\usepackage{geometry}
\usepackage{fancyhdr}
\usepackage{fp}

\newtheorem{theorem}{Theorem}[section]
\newtheorem{proposition}[theorem]{Proposition}
\newtheorem{corollary}[theorem]{Corollary}
\newtheorem{lemma}[theorem]{Lemma}
\newtheorem{remark}[theorem]{Remark}
\newtheorem{definition}[theorem]{Definition}

\numberwithin{equation}{section}

\newcommand{\pa}{\partial}

\renewcommand{\a }{\alpha }
\renewcommand{\b }{\beta }
\newcommand{\g }{\gamma}

\renewcommand{\l }{\lambda }

\renewcommand{\t }{\tau }

\newcommand{\R}{{\mathbb{R}}}

\newcommand{\LL}{\mathcal{L}}

\title{Uniqueness for local-nonlocal elliptic equations}

\author[S.\,Biagi]{Stefano Biagi} 
\author[G.\,Meglioli]{Giulia Meglioli} 
 \author[F.\,Punzo]{Fabio Punzo}
 
  \address[S.\,Biagi]{Dipartimento di Matematica
 \newline\indent Politecnico di Milano \newline\indent
 Via Bonardi 9, 20133 Milano, Italy}
 \email{stefano.biagi@polimi.it}
 
 \address[G.\,Meglioli]{Fakult\"at f\"ur Mathematik
 \newline\indent Universit\"at Bielefeld, \newline\indent
 33501, Bielefeld, Germany}
 \email{gmeglioli@math.uni-bielefeld.de}
 
 \address[F.\,Punzo]{Dipartimento di Matematica
 \newline\indent Politecnico di Milano \newline\indent
 Via Bonardi 9, 20133 Milano, Italy}
 \email{fabio.punzo@polimi.it}



\date{}

\begin{document}
\begin{abstract}
We study mixed local and nonlocal elliptic equation with a variable coefficient $\rho$. Under suitable assumptions on the behaviour at infinity of $\rho$, we obtain uniqueness of solutions belonging to certain weighted Lebsgue spaces, with a weight depending on the coefficient $\rho$. The hypothesis on $\rho$ is optimal; indeed, when it fails we get nonuniqueness of solutions. We also investigate the parabolic counterpart of such equation.
\end{abstract}
\maketitle

\section{Introduction}
We are concerned with uniqueness of solutions, in suitable weighted Lebesgue spaces, to the following linear, local and nonlocal elliptic equation
\begin{equation}\label{elliptic}
\LL u -\rho(x)cu = \Delta u -(-\Delta)^{s} u - \rho c  u \,=\,0 \quad \textrm{in}\;\; \R^N.
\end{equation}
Here $\rho$ is a positive function depending only on the space variable, $(-\Delta)^s$ denotes the fractional Laplace operator of order $s\in (0,1)$ and $c$ is a nonnegative function defined in $\R^N$. Sometimes, such type of equations are called "mixed" due to the fact that the operator 
$\LL$ combines classical and fractional features.

\smallskip

We always assume that
\begin{equation}\tag{{\it $H_0$}}\label{h3}
\begin{aligned}
&\textrm{(i)} \,\rho\in C(\R^N);\\
&\textrm{(ii)} \;\text{there exist}\,\,\alpha\ge 0\,\,\text{and}\,\,C_0>0\,\,\text{such that}\,\,\\
&\quad \; \rho(x)\geq C_0 (1+|x|^2)^{-\frac{\alpha}2} \quad\text{for all}\,\,x\in \mathbb R^N;
\end{aligned}
\end{equation}
and that the function $c:\R^N\to\R$ is
\begin{equation}\label{h1}\tag{{\it $H_1$}}
c\in C(\R^N),\,\,\,c(x)\geq 0 \quad \textrm{for all}\;\; x\in \R^N\,.
\end{equation}

The operator appearing in equation \eqref{elliptic} is enjoying a constantly rising popularity in applied sciences, also with the aim of investigating the different role of a local and a nonlocal diffusion in concrete situations, see e.g. \cite{BD,EG}.
In parallel to the applied sciences, also in the mathematical framework, such kind of operators are more often considered and indeed the literature already presents many results. Here we cite some contribution without entering into further details: \cite{BDVV2,BDVV,BDVV4,BDVV3,BMV,BJK,CKSV,CKSV2,C, DFM,DEJ,JS}.
\medskip

To the best of our knowledge, this type of uniqueness results for equation \eqref{elliptic} which involves a mixed local and nonlocal operator were never obtained in the literature. On the other hand, there exists several results in the literature which deals with local or nonlocal elliptic equations. We now briefly recall some of such known results. In the framework of local operators, let us consider the following linear, elliptic, degenerate equation
\begin{equation}\label{02}
\it{L}u-cu=f \quad\quad\text{in}\,\,\,\Omega\,;
\end{equation}
where
$$
\it{L} u:= \sum_{i,j=1}^Na_{ij}\frac{\partial^2  u }{\pa x_i\pa x_j} +\sum_{i=1}^Nb_i\frac{\pa u}{\pa x_i}\,,
$$
with possibly unbounded coefficients $a_{ij}$, $b_i$, $c$ and function $f$.
Well-posedness of problem \eqref{02} has been intensively investigated both in the case of $\Omega=\R^N$ and $\Omega$ being a bounded domain, in this last case the problem is completed with suitable boundary conditions and the coefficients can be degenerate or singular at the boundary of the domain (see e.g. \cite{OR,P2,PTes} and the references therein). Furthermore, uniqueness of solutions to equations like
\begin{equation}\label{03}
-\Delta u - g( b,\,\nabla u ) +c\,u=0 \quad  \text{in}\,\,\,M,
\end{equation}
has been studied also on Riemannian manifolds, here we denote with $M$ a complete, non-compact, Riemannian manifold of infinite volume. More precisely, in \cite{Grig3} the author addresses the case of bounded solutions for $b\equiv0$, instead in \cite{MR} solutions are assumed to belong to a certain weighted $L^p$ space. The parabolic counterparts of \eqref{02} and \eqref{03} have also a wide background. We refer the interested reader to e.g \cite{AB,EKP,MontPun,PPT2,PoT,PTes} for problem posed on $\mathbb{R}^N$ and in bounded domains of $\mathbb{R}^N$ or to e.g. \cite{Grig} in the case of more general Riemannian manifolds.

\bigskip

Similarly, the case of well-posedness and regularity of solutions to equations with fractional diffusion has been addressed e.g. in \cite{ MePu,PV1,PV2}. In particular, in \cite{PV1} and in \cite{PV2}, it is studied uniqueness of solutions to the following elliptic equation
\begin{equation}\label{eq120}
(-\Delta)^s u +\rho\,cu  =0 \quad\text{in}\,\,\,\R^N\,.
\end{equation}
The authors investigate how $\rho$ affects uniqueness and nonuniqueness of solutions. Uniqueness is obtained for solutions belonging to $L^p_\psi(\R^N)$, for a suitable weight $\psi$, depending on the behavior of $\rho$ as $|x|\to +\infty$. 
In \cite{MePu}, equation \eqref{eq120} has been investigated in presence of a drift term, i.e.
\begin{equation}\label{eq120b}
(-\Delta)^s u  +\langle b,\,\nabla u\rangle+\rho\,cu=0 \quad  \text{in}\,\,\,\R^N,\\
\end{equation}
where $b:\R^N\to\R^N$ is a suitable given vector field. The authors show that the solution to  equation \eqref{eq120b} is unique in the class $L^p_{\psi}(\R^N)$ with $p\geq 1$ and
\begin{equation*}
\psi(x):=(1+|x|^2)^{-\frac{\beta}2}\;\,\; (x\in \R^N)\,.
\end{equation*}
for properly chosen $\beta>0$, provided that there exist $\sigma\le 1-\alpha$ and $K>0$ such that
$$
\left\langle b(x), \frac{x}{|x|}\right\rangle\le K(1+|x|)^{\sigma}\quad\text{for all}\,\,x\in \R^N.
$$
Moreover, also uniqueness of solutions to the parabolic counterpart of \eqref{eq120} and \eqref{eq120b} have been investigated in \cite{PV1,PV2,MePu} and similar results have been obtained.

\subsection{Outline of our results}
The main results of this paper will be given in detail in the forthcoming Theorems \ref{teo3} and \ref{cor4}. We give here a sketchy outline of these results.

\medskip

We shall prove that the solution to the elliptic equation \eqref{elliptic} is unique in the class $L^p_{\psi}(\R^N)$ with $p\geq 1$ and
\begin{equation}\label{eq12}
\psi_\beta(x):=(1+|x|^2)^{-\frac{\beta}2}\;\,\; (x\in \R^N)\,,
\end{equation}
for properly chosen $\beta>0$, see Theorem \ref{teo3}.
In order to prove such a uniqueness result we construct a positive supersolution to equation
\begin{equation}\label{eq14}
\LL\zeta -\rho\,p\,c\,\zeta= 0\quad \textrm{in}\;\; \R^N\,.
\end{equation}
Indeed, the weight function $\psi$ defined in \eqref{eq12} is related to such a supersolution. In general, our uniqueness class includes {\it unbounded} solutions. Thus, in particular, we get uniqueness of bounded solutions. Moreover, we also show that our conditions under which we prove uniqueness are sharp in the sense that, when they fail, infinitely many solutions can be found, see Theorem \ref{cor4}.


\smallskip

\medskip

The paper is organized as follows. In Section \ref{sec1} we recall some preliminaries about fractional Laplacian and we give the notion of solutions we shall deal with. Then we state our main results (Theorems \ref{teo3}, \ref{cor4}). Section \ref{pe} is devoted to the proof of Theorem \ref{teo3}. In Section \ref{mix} we state some preliminaries concerning regularity of solutions to mixed local and nonlocal equations and we conclude by proving Theorem \ref{cor4}. Finally, in Section \ref{pp}, we also establish similar uniqueness results for the parabolic problem \eqref{problema} (see Theorem \ref{teo1} below). 

\section{Mathematical background and main results} \label{sec1}\setcounter{equation}{0}
In this first section we collect the relevant definitions and notation which will
be used throughout the rest of the paper, and we state our main result,
namely Theorems \ref{teo3} and \ref{cor4}.
\subsection{The mixed operator $\LL = \Delta-(-\Delta)^s$}
In order to clearly state the main results of this paper, we first need
to fix some notation and to pro\-perly define
what we mean by a \emph{solution of equation \eqref{elliptic}}; due to the
\emph{mixed nature of $\LL$}, this will require some preliminaries.
\medskip

\noindent\textbf{1) The Fractional Laplacian.} Let $s\in (0,1)$ be fixed, and let
$u:\R^N\to\R$. The \emph{fractional La\-pla\-cian} (of order $s$) of $u$
at a point $x\in\R^N$ is defined as follows
\begin{equation} \label{eq:defDeltas}
\begin{split}
 (-\Delta)^s u(x) & = C_{N,s}\cdot \mathrm{P.V.}\int_{\R^N}\frac{u(y)-u(x)}{|x-y|^{N+2s}}\,dy
\\
& = C_{N,s}\cdot\lim_{\varepsilon\to 0^+}\int_{\{|x-y|\geq\varepsilon\}}\frac{u(y)-u(x)}{|x-y|^{N+2s}}\,dy,
\end{split}
\end{equation}
provided that the limit exists and is finite.
Here, $C_{N,s} > 0$ is a suitable normalization constant which plays
a role in the limit as $s\to 0^+$ or $s\to 1^-$, and is explicitly given by
$$C_{N,s} = \frac{2^{2s-1}{2s}\Gamma((N+2s)/2)}{\pi^{N/2}\Gamma(1-s)}.$$
As is reasonable to expect, for $(-\Delta)^s u(x)$ to be well-defined one needs
to impose suitable \emph{growth conditions} on the functions $u$, both when $y\to\infty$ and
when $y\to x$. In this perspective we state the following
proposition, where we employ the notation
\begin{equation}\label{eq:spaceL1s}
 \mathcal{L}^s(\R^N) := 
 \Big\{f:\R^N\to\R:\,\|u\|_{1,s} := \int_{\R^N}\frac{|f(x)|}{1+|x|^{N+2s}}\,dx<\infty\Big\}.
\end{equation}
\begin{proposition} \label{prop:welldefDeltas}
Let $\Omega\subseteq\R^N$ be an open set. Then, the following facts hold.
\begin{itemize}
 \item[{i)}] If $0<s<1/2$ and $u\in C_{\mathrm{loc}}^{2s+\g}(\Omega)\cap \mathcal{L}^s(\R^N)$
 for some $\g \in (0,1-2s)$, then
 $$\exists\,\,(-\Delta)^s u(x) = C_{N,s}\,\int_{\R^N}\frac{u(y)-u(x)}{|x-y|^{N+2s}}\,dy\quad
 \text{for all $x\in\Omega$}. $$
 \item[{ii)}] If $1/2<s<1$ and $u\in C^{1,2s-1+\g}_{\mathrm{loc}}(\Omega)\cap \mathcal{L}^s(\R^N)$
 for some $\g\in (0,2-2s)$, then
 $$\exists\,\,(-\Delta)^s u(x) = -\frac{C_{N,s}}{2}\,\int_{\R^N}\frac{u(x+z)+u(x-z)-2u(x)}{|z|^{N+2s}}\,dy\quad
 \text{for all $x\in\Omega$}. $$
\end{itemize}
Moreover, in both cases \emph{i)-ii)} we have $(-\Delta)^su\in C(\Omega)$.
\end{proposition}
In the particular case when $\Omega = \R^N$
and $u\in\mathcal{S}\subseteq\mathcal{L}^s(\R^N)$ 
(here and throughout, $\mathcal{S}$ denotes the usual Schwarz space of rapidly decreasing functions),
it is possible to provide an alternative expression of $(-\Delta)^s u$ (which is well-defined
on the whole of $\R^N$, see Proposition \ref{prop:welldefDeltas}) via the Fourier Transform
$\mathfrak{F}$; more precisely, we have the following result.
\begin{proposition} \label{prop:DeltasFourier}
 Let $u\in\mathcal{S}\subseteq \mathcal{L}^s(\R^N)$. Then,
 \begin{equation} \label{eq23}
  \exists\,\,(-\Delta)^s u(x) = 
  \mathfrak F^{-1} \big( |\xi|^{2s}\mathfrak F u  \big)(x)\quad \text{for every
  $x\in\R^N$}.
 \end{equation}
\end{proposition}
\begin{remark} \label{rem:hintweakLap}
 Taking into account that the Fourier Transform $\mathfrak{F}$ 
 \emph{can be extended to
  an isometry of $L^2(\R^N)$}, identity \eqref{eq23} in the above proposition
  may be used to define the fractional Laplacian for \emph{non-regular functions}; more precisely,
  one can define
  $$(-\Delta)^s u = \mathfrak F^{-1} \big( |\xi|^{2s}\mathfrak F u  \big)\in L^2(\R^N)$$
  for every $u\in L^2(\R^N)$ such that $ |\xi|^{2s}\mathfrak F u \in L^2(\R^N)$.
  We will return on the proper definition of a \emph{weak} fractional Laplacian
  in Appendix \ref{sec:appendix}.
\end{remark}
\medskip

\noindent \textbf{2)\,\,Classical solutions to $\LL u = f$.} Now we have reviewed some basic
facts concerning the fr\-ac\-tional Laplacian, we can make precise the notion
of \emph{classical solution} to equation \eqref{elliptic}.
\begin{definition}\label{defsole}
 Let assumptions $(H_0)-(H_1)$ be in force. We say that a function $u:\R^N\to\R$ is a 
 \emph{classical subsolution [\emph{resp.}\,{supersolution}]} to equation \eqref{elliptic} if
 \begin{itemize}
  \item[	{a)}] $u\in C^2(\R^N)\cap \mathcal{L}^s(\R^N)$;
  \item[{b)}] for every $x\in\R^N$, we have
  $$\mathcal{L} u(x) -\rho(x)c(x)u(x)\leq\,[\text{resp.}\,\geq]\,\,0.$$
 \end{itemize}
 We say that $u$ is a \emph{classical solution} to \eqref{elliptic} if $u$ is both
 a classical subsolution and a classical supersolution to the same equation.
\end{definition}
We explicitly observe that the well-posedness of Definition \ref{defsole} is
a direct consequence of Proposition \ref{prop:welldefDeltas}: in fact,
if $u\in C^2(\R^N)\cap \mathcal{L}^s(\R^N)$ (for some fixed $s\in (0,1)$), we have
$$\exists\,\,(-\Delta)^s u(x) = C_{N,s}\cdot \mathrm{P.V.}\int_{\R^N}\frac{u(y)-u(x)}{|x-y|^{N+2s}}\,dy\quad
\text{for every $x\in\R^N$},$$
and the function $\R^N\ni x\mapsto (-\Delta)^s u(x)$ is \emph{continuous in $\R^N$}.
\medskip

\noindent\textbf{3) Weighted $L^p$-spaces.} Since the main aim of this paper
is to obtain some \emph{uniqueness re\-sults} for the clasical solutions of \eqref{elliptic} belonging to 
suitable \emph{weighted
$L^p$-spaces}, we conclude this part of the section with the following definition/notation.
\begin{definition} \label{def:wLpspaces}
 Let $1\leq p\leq \infty$, and let $f\in L^1(\R^N),\,f\geq 0$. We define
 \begin{align*}
 & \bullet\,\,L^p_f(\R^N) = \Big\{u:\R^N\to\R:\,\int_{\R^N}|u|^pf\,dx < \infty\Big\}
 \quad\text{if $1\leq p <\infty$};
 \\[0.1cm]
 & \bullet\,\,L^\infty_f(\R^N) = \big\{
 u:\R^N\to\R:\,\mathrm{ess\,sup}_{\R^N}(|u|f) < \infty\}.
 \end{align*} 
 These spaces $L^p_f(\R^N)$ are usually referred to as \emph{weighted $L^p$-spaces
 (with weight $f$)}.
\end{definition}
\subsection{Main results}

\begin{theorem}\label{teo3}
Let assumptions \eqref{h3}-\eqref{h1} be satisfied. Let $u
\in C^2(\R^N)\cap\mathcal{L}^s(\R^N)$ 
be a solution to equation \eqref{elliptic} with $|u|^p\in \mathcal L^s(\R^N),$ for some $p\geq 1$. Suppose that, for some $c_0>0,$
\begin{equation}\label{eq27}
c(x)\geq c_0\quad \textrm{for all}\;\;x\in \R^N\,.
\end{equation}
If $pc_0$ is large enough and if $u\in L^p_\psi(\R^N)$ 
\emph{(}where $\psi = \psi_\beta$ is as in \eqref{eq12}\emph{)}, then
$$
u\equiv 0 \quad \textrm{in} \;\; \R^N\,.
$$
provided that one of the following condition is fulfilled:
\begin{itemize}
\item[(i)] $0<\beta\leq N-2s$, $\alpha\leq 2$;
\item[(ii)] $N-2s<\beta<N$, $\alpha\leq 2s$;
\item[(iii)] $\beta=N$, $\alpha<2s$;
\item[(iv)] $\beta>N$, $\alpha+\beta\leq 2s +N$.
\end{itemize}
\end{theorem}

\begin{remark}
{\rm The hypothesis $p c_0 $ large enough made in Theorem \ref{teo3} will be specified in the proof of Theorem \ref{teo3}.}
\end{remark}

\begin{remark}\label{remun}
\rm{Let us assume that
$$
|u(x)|\le C(1+|x|^2)^{\frac{\vartheta}{2}} \quad \textrm{for all}\;\; x\in \R^N,
$$
for some $\vartheta\in\R$, $C>0$. Then  $u\in L^p_{\psi_\beta}(\R^N)$, with $p\in[1,+\infty)$, if $\vartheta<\frac{\beta-N}{p}$. Thus, in particular, a bounded solution $u$ belongs to our class of uniqueness $L^p_{\psi_\beta}(\R^N)$, with $p\in[1,+\infty)$, only if condition $\rm{(iv)}$ of Theorem \ref{teo3} is fulfilled.}

\end{remark}

In particular, Theorem \ref{teo3} with $\alpha=0$ in \eqref{h3}, yields the following

\begin{corollary}\label{corpar}
Let assumption \eqref{h3} be satisfied with $\alpha=0$, and let $u$ be a solution to problem \eqref{elliptic} with $|u|^p\in \mathcal L^s(\R^N)$, for some $p\geq 1$. Assume that \eqref{eq27} is satisfied. If $p c_0$ is large enough and if $u\in L^p_{\psi}(\R^N)$, then
$u\equiv 0$ in $\R^N$ proved that
$$
0< \beta\le N+2s.
$$
\end{corollary}

From Theorem \ref{teo3} we also deduce the following

\begin{corollary}\label{cor3}
Let assumptions $(H_0)-(H_1)$ be satisfied, and let $u\in C^2(\R^N)\cap \mathcal{L}^s(\R^N)$ 
be a so\-lu\-tion to equation \eqref{elliptic}.
We assume that $\alpha < 2s$, and that 
\begin{equation} \label{eq:growthuCorollary}
 |u(x)| \leq C(1+|x|^2)^{\frac{\vartheta}{2}}\quad \textrm{for all}\;\; x\in \R^N,
 \end{equation}
for some $C>0$ and $0\leq \vartheta < 2s-\alpha$.
If $pc_0$ is large enough, with $c_0$ is as in \eqref{eq27}, then
$$u\equiv 0\quad \textrm{in}\;\; \R^N\,.$$
\end{corollary}
\begin{proof}
 It suffices to apply Theorem \ref{teo3}-(iv), since
 the growth assumption
 \eqref{eq:growthuCorollary} (jointly with the
 fact that $\vartheta < 2s-\alpha$) ensures that $u\in L^1_\psi(\R^N)$ for any 
 $\vartheta+N<\beta\leq (2s-\alpha)+N$.
\end{proof}
Moreover, we show that hypothesis $\alpha < 2s$ in Corollary \ref{cor3} is 
\emph{almost optimal}. More precisely, 
 the next result shows that problem \eqref{problema} admits \emph{infinitely many} bounded solutions
 if $\alpha > 2s$.
 
\medskip

\begin{theorem}\label{cor4} Let $\rho, c\in C^\gamma(\mathbb R^N)$, for some $\gamma\in (0,1)$, and suppose that \eqref{h1} holds. 
Moreover, let us suppose that the function $\rho$ satisfies, for some $c_0,r_0 > 0$, the bound
$$0< \rho(x)\leq \frac{c_0}{(1+|x|^2)^{\alpha/2}}\quad\text{for every $x\in\R^N$ with $|x|>r_0$},$$
and suppose that $\alpha > 2s$. Then equation \eqref{elliptic} admits infinitely many bounded solutions. 
\end{theorem}

 \begin{remark}\label{Remnonun}
 \rm{Note that we get non-uniqueness of bounded solutions. Our constraint on $\alpha$ in Theorem \ref{cor4} is compatible with the uniqueness result of Theorem \ref{teo3} case $(i)$. In fact, in that case, bounded solutions are not dealt with. Hence the two results are not in contrast. }
 \end{remark}



%

\section{Elliptic equations: proofs}\label{pe}\setcounter{equation}{0}

Let us observe that
\begin{equation*}
\begin{aligned}
&\text{if}\,\, f,g  \in \mathcal L^s(\R^N)\cap C_{loc}^{2s+\gamma}(\R^N)\,\,\,\,\quad\quad \text{with}\,\,\, s<\frac 1 2, \\
&\text{or}\,\,f,g\in \mathcal L^s(\R^N)\cap C^{1,2s+\gamma-1}_{loc}(\R^N)\,\,\,\, \text{with}\,\,\, s\geq \frac 1 2,
\end{aligned}
\end{equation*}
for some $\gamma>0$, and $fg\in \mathcal L^s(\R^N)$, then it is easily checked that
\begin{equation}\label{eq30}
(-\Delta)^s[f(x)g(x)]= f(x)(-\Delta)^{s}g(x) + g(x)(-\Delta)^sf(x)- \mathcal B(f,g)(x),
\end{equation}
for all $x\in \R^N\,$, where $\mathcal B(f,g)$ is the bilinear form given by
$$
\mathcal B(f,g)(x):=C_{N,s}\int_{\R^N}\frac{[f(x)-f(y)][g(x)-g(y)]}{|x-y|^{N+2s}}dy \quad \textrm{for all}\;\; x\in \R^N\,.
$$
\medskip

Let us now state a general criterion for uniqueness of nonnegative solutions to equation \eqref{elliptic} in $L^1_{\psi}(\R^N)$. We will use this result as a key tool for proving Theorem \ref{teo3}.

\begin{proposition} \label{prop3}
Let assumptions \eqref{h3}-\eqref{h1} be satisfied, and let $u\in C^2(\R^N)\cap\mathcal{L}^s(\R^N)$ 
be a solution to equation 
\eqref{elliptic} such that $|u|^p\in \mathcal L^s(\R^N),$ for some $p\geq 1$. We assume that there exists a 
\emph{positive function} $\zeta\in C^2(\R^N)\cap \mathcal{L}^s(\R^N)$ such that
\begin{align}
 &\mathrm{a)}\,\,\LL\zeta  - p\rho(x)\,c \,\zeta < 0 \quad \textrm{in $\R^N$}; \label{eq40} \\
 &\mathrm{b)}\,\,\zeta(x)+|\nabla \zeta(x)|\leq C \psi(x)\quad \textrm{for all}\;\; x\in \R^N.  \label{eq42d}
\end{align}
for some constant $C>0$ and for $\psi$ as in \eqref{eq12}. If $u\in L^p_\psi(\R^N)$, then
$$
u\equiv 0\quad \textrm{in}\;\; \R^N\,.
$$
\end{proposition}

\subsection{Proof of Proposition \ref{prop3}}

Take a cut-off function $\gamma\in C^\infty([0,\infty)), 0\leq \g\leq1$ with
\begin{equation}\label{eq31}
\gamma(r)=
\begin{cases}
1 &\textrm{if}\,\,0\leq r\leq \frac 1 2\\
0& \textrm{if\ \ } r\geq 1 \,
\end{cases};
\quad\quad\quad \gamma'(r)<0.
\end{equation}
Moreover, for any $R>0$ let
\begin{equation}\label{eq31bbbb}
\gamma_R(x):= \gamma\left(\frac{|x|}{R}\right) \quad \textrm{for all}\;\; x\in \R^N\,.
\end{equation}
To prove Proposition \ref{prop3} we shall use the next two results.

\begin{lemma}\label{lemma4}
Let assumption \eqref{h1} be satisfied, and let $\zeta\in C^2(\R^N)\cap \mathcal{L}^s(\R^N)$ be a \emph{positive
fun\-ction} in $\R^N$ satisfying \eqref{eq42d}. If $v\in L^1_{\psi}(\R^N)$, we have
\begin{equation}\label{eq41bis}
\int_{\R^N}|v(x)|\zeta(x)|(-\Delta)^s\gamma_R(x)|\,dx + \int_{\R^N}|v(x)|\,|\mathcal B(\zeta, \gamma_R)(x)|dx \longrightarrow 0,
\end{equation}
and
\begin{equation}\label{eq41tris}
\int_{\R^N}|v(x)|\zeta(x)|\Delta\g_R(x)|\,dx \,+\int_{\R^N}|v(x)|\,|\left\langle \nabla\zeta(x), \nabla\g_R(x)\right\rangle| dx\longrightarrow 0,
\end{equation}
as $R\to \infty$.
\end{lemma}

Observe that a similar result was obtained in \cite{MePu, PV1}. The novelty in Lemma \ref{lemma4} is the limit in \eqref{eq41tris}. For this reason, the proof of Lemma 
\ref{lemma4} is a slight modification of the proof of \cite[Lemma 3.1]{PV1}, hence we just show how to treat the extra term given by $-\Delta$.

\medskip

\begin{proof}[Proof of Lemma \ref{lemma4}]
In view of \eqref{eq42d}, by arguing as in the proof of \cite[Lemma 3.1]{PV1}, we get
\begin{equation}\label{eq33a}
\begin{aligned}
&\int_{\R^N}|v(x)|\zeta(x)|(-\Delta)^s\g_R(x)|\,dx  \longrightarrow 0\\
&\int_{\R^N}|v(x)|\,|\mathcal B(\zeta, \g_R)(x)|\,dx  \longrightarrow 0,
\end{aligned}
\end{equation}
as $R\to \infty$, hence \eqref{eq41bis} is proved.

\noindent To show \eqref{eq41tris}, let us observe that, due to \eqref{eq31} and \eqref{eq31bbbb}, for any $x\in B_R\setminus B_{R/2}$, for some $\bar C>0$
\begin{equation}\label{eq33b}
\begin{aligned}
&|\Delta \gamma_R|\le\frac{\bar C}{R^2}\,,\\
&|\left\langle \nabla \zeta(x),\nabla \gamma_R(x)\right\rangle|\le |\nabla\zeta(x)||\nabla\g_R|,\le \frac{\bar C}{R}\,|\nabla\zeta(x)|.
\end{aligned}
\end{equation}
 Then, due to \eqref{eq42d} and \eqref{eq33b}, we get
 \begin{equation}\label{eq33bis}
\begin{aligned}
\int_{\R^N} |v(x)|\zeta(x)|\Delta \gamma_R|\,dx  &\le \frac{\bar C}{R^2} \int_{B_R\setminus B_{R/2}} |v(x)|\zeta(x)\,dx \\
&\le \frac{\bar C C}{R^2} \int_{B_R\setminus B_{R/2}} |v(x)| \psi(x)\,dx,
\end{aligned}
\end{equation}
and
\begin{equation}\label{eq33}
\begin{aligned}
\int_{\R^N} |v(x)|\,|\left\langle \nabla \zeta(x),\nabla \gamma_R(x)\right\rangle|\,dx  &\le \frac{\bar C}{R}\,\int_{B_R\setminus B_{R/2}} |v(x)| \,|\nabla\zeta(x)|\,dx \\
&\le \frac{\bar C C}{R}\, \int_{B_R\setminus B_{R/2}} |v(x)| \psi(x)\,dx.
\end{aligned}
\end{equation}
Now, since $v\in L^1_{\psi}(\R^N)$, we obtain from \eqref{eq33} that
$$
\int_{\R^N}|v(x)|\zeta(x)|\Delta\g_R(x)|\,dx +\int_{\R^N}|v(x)|\,|\left\langle \nabla\zeta(x), \nabla\g_R(x)\right\rangle| dx \longrightarrow 0,
$$
as $R\to\infty$. This completes the proof.
\end{proof}

We use the next lemma (see, e.g., \cite[Lemma 3.4]{MePu}).
\begin{lemma}\label{lemma2}
Let $G\in C^2(\R)$ be a convex function, and let $u\in C^2(\R^N)\cap \mathcal{L}^s(\R^N)$.
We assume that $G(u)\in \mathcal L^s(\R^N)$. Then, we have the following inequality
\begin{equation}\label{eq35}
(-\Delta)^s[G(u)]\leq G'(u)(-\Delta)^s u\quad \textrm{in}\,\;\R^N\,.
\end{equation}
\end{lemma}
We are now ready to provide the
\begin{proof}[Proof of Proposition \ref{prop3}] Take a \emph{non-negative} 
function $v\in C^2(\R^N)$ with \emph{compact support}. 
Moreover, take a function $w\in \mathcal L^s(\R^N)\cap C^2(\R^N)$.
Integrating by parts we have:
\begin{equation}\label{eq41}
\begin{aligned}
\int_{\R^N} &v \left[\LL w  - \rho(x)  c(x) w \right]\,  dx  
 = \int_{\R^N} w\left[\LL v  - \rho(x)  c(x) v\right]\, dx.
\end{aligned}
\end{equation}
Let $p\geq 1.$ For any $\a>0$, set
\begin{equation}\label{eq37}
G_\alpha(r):=(r^2+\alpha)^{\frac p 2} \quad \textrm{for all}\;\; r\in\R\,.
\end{equation}
It is easily seen that
\begin{equation}\label{eq38}
\begin{aligned}
&G_\alpha'(r)=pr(r^2+\alpha)^{\frac p2 -1} \\
&G_\alpha''(r)=p(r^2+\alpha)^{\frac p2 -2}[\alpha+r^2(p-1)]\geq 0 \quad \textrm{for all}\;\; r\in\R\,.
\end{aligned}
\end{equation}
Therefore, observe that
\begin{equation}\label{eq39bis}
\Delta[G_{\alpha}(u)]=G_{\alpha}'(u)\Delta u + G_{\alpha}''(u)|\nabla u|^2\quad \text{in}\,\,\R^N.
\end{equation}
Due to Lemma \ref{lemma2} and \eqref{eq39bis}, and since $u$ solves \eqref{elliptic}, we obtain
\begin{equation}\label{eq42}
\begin{aligned}
 \LL G_\alpha(u) - \rho  c \,G_\alpha(u) &\geq  G_{\alpha}'(u)\Delta u + G_{\alpha}''(u)|\nabla u|^2 
  - G_{\alpha}'(u)(-\Delta)^su - \rho c \,G_\alpha(u) \\
& =  p\,u (u^2+\alpha)^{\frac p2-1}\Delta u + p(u^2+\alpha)^{\frac p2-2}[u^2(p-1)+\alpha]\,|\nabla u|^2 \\
&\quad\quad-p\,u (u^2+\alpha)^{\frac p2-1}(-\Delta)^s u-\rho\, c\, G_{\alpha}(u) \\
&\quad\quad +  p\rho\,  c\,u^2(u^2+\alpha)^{\frac p2-1} - p\rho\, c\,u^2(u^2+\alpha)^{\frac p2-1}\\
&\geq p\,u (u^2+\alpha)^{\frac p2 -1}[\LL u- \rho\,  c\, u] \\
&\quad\quad -  \rho\, c(u^2+\a)^{\frac p2-1}\left[u^2+\alpha-p\,u^2\right] \\
&=\rho\, c(u^2+\a)^{\frac p2-1}\left[u^2(p-1)-\alpha\right]\quad \quad\quad\textrm{in}\;\; \R^N\,.
\end{aligned}
\end{equation}
From \eqref{eq41} with $w=G_\alpha(u)$ and \eqref{eq42} it follows that
\begin{equation}\label{eq43}
\begin{aligned}
\int_{\R^N} v (u^2+\alpha)^{\frac p2-1} &\rho(x) c(x)[(p-1)u^2- \alpha] \,dx\\
&\leq \int_{\R^N} G_\alpha(u)\big[ \LL v -\rho(x) c\, v\big]\, dx.
\end{aligned}
\end{equation}
Letting $\a\to 0^+$ in \eqref{eq43}, by the dominated convergence
theorem we get
\begin{equation}\label{eq44}
\int_{\R^N} |u|^p\big[\LL v - p\,\rho\, c\, v \big]\, dx\,\geq \,0.
\end{equation}
For any $R>0$, we now choose
\[v(x):= \zeta(x)\g_R(x)\quad \textrm{for all}\;\; x\in \R^N\,,\]
where $\zeta$ satisfies assumption \eqref{eq40}. Clearly, $v$ is a \emph{non-negative} $C^2$-function
in $\R^N$ with compact support; moreover, a direct computation based on \eqref{eq30} gives
\begin{equation}\label{eq45}
\begin{aligned}
\LL v - \,p\,\rho\, c v&=\gamma_R \left[\LL \zeta  - p\,\rho\, c \,\zeta\right]
+\zeta \Delta\gamma_R 
+2\langle \nabla\zeta\,,\,\nabla \gamma_R\rangle 
 - \zeta(-\Delta)^s\gamma_R +\mathcal B(\zeta,\gamma_R).
\end{aligned}
\end{equation}
Thus, by combining \eqref{eq44}-\eqref{eq45}, we obtain
\begin{equation}\label{eq46}
\begin{aligned}
\int_{\R^N} |u|^p \gamma_R& \left[\LL\zeta - p\, \rho\, c \,\zeta\right]\, dx \\
&\ge -\int_{\R^N} |u|^p \left[\zeta \Delta\gamma_R +2\,\langle \nabla\zeta\,,\,\nabla \gamma_R\rangle-\zeta(-\Delta)^s\gamma_R +\mathcal B(\zeta,\gamma_R) \right]\,dx\\
&= -\int_{\R^N}|u|^p\left[\zeta \Delta\gamma_R +2\langle \nabla\zeta\,,\,\nabla \gamma_R\rangle\right] \,dx\\
&\qquad -\int_{\R^N} |u|^p \left[-\zeta(-\Delta)^s\gamma_R +\mathcal B(\zeta,\gamma_R)\right] \,dx\,.
\end{aligned}
\end{equation}
With \eqref{eq46} at hand, we can easily complete the proof of the proposition: in fact,
since $\zeta$ satisfies \eqref{eq42d}, we are entitled apply Lemma \ref{lemma4}: this, together with
 the monotone convergence theorem, allows us to pass to the limit as $R\to+\infty$ in \eqref{eq46},
 obtaining
\begin{equation}\label{eq46a}
\int_{\R^N} |u|^p \left[\LL\zeta - p\, \rho\, c \,\zeta\right]\,dx\geq 0\,.
\end{equation}
From \eqref{eq46a} and \eqref{eq40}, since $|u|^p\geq 0$, we can infer that $u\equiv 0$ in $\R^N$. This completes the proof.
\end{proof}

\medskip

\subsection{Proof of Theorem \ref{teo3}}

Before proving Theorem \ref{teo3}, we need some preliminary results. Observe that the proof of Proposition \ref{prop5} can be found in \cite[Proposition 3.3]{PV1}.

\begin{proposition}\label{prop5}
Let $\tilde w\in C^2([0,\infty))\cap L^\infty((0,\infty)).$ Let
$$w(x):=\tilde w(|x|)\quad \textrm{for all}\;\; x\in \R^N\,.$$
Set $r\equiv |x|$. If
\begin{equation}\label{eq317}
\tilde w''(r)+ \frac{N-2s+1}{r}\tilde w'(r)\leq 0,
\end{equation}
then $w$ is a supersolution to equation
\begin{equation}\label{eq318}
(-\Delta)^s w\,=\,0\quad \textrm{in}\;\; \R^N\,.
\end{equation}
\end{proposition}

\medskip

In the sequel we shall use the next well-known result, concerning the hypergeometric function $_{2}F_1(a,b,c,s)\equiv F(a,b,c,s)$, with $a,b\in \R, c>0, s\in \R\setminus\{1\}$  (see \cite[Chapters 15.2, 15.4]{DLMF}).

\begin{lemma}\label{lemma3}
The following limits hold true:
\begin{itemize}
\item[(i)] if $c>a+b$, then
\[\lim_{s\to 1^-} F(a,b,c, s)=\frac{\Gamma(c)\Gamma(c-a-b)}{\Gamma(c-a)\Gamma(c-b)}\,;\]
\item[(ii)] if $c=a+b$, then
\[\lim_{s\to 1^-}\frac{F(a,b,c,s)}{-\log(1-s)}=\frac{\Gamma(a+b)}{\Gamma(a)\Gamma(b)}\,;\]
\item[(iii)] if $c<a+b$, then
\[\lim_{s\to 1^-}\frac{F(a,b,c,s)}{(1-s)^{c-a-b}}=\frac{\Gamma(c)\Gamma(a+b-c)}{\Gamma(a)\Gamma(b)}\,.\]
\end{itemize}
\end{lemma}
For further references, observe that
\begin{equation}\label{eq319}
\Gamma(t)>0\quad\textrm{for all}\;\; t>0, \quad \Gamma(t)<0\quad \text{for all}\;\; t\in (-1,0)\,.
\end{equation}
For the proof of Lemma \ref{lemma3}, we refer the reader to \cite[Chapters 15.2, 15.4]{DLMF}.
\bigskip

\begin{proof}[Proof of Theorem \ref{teo3}\,.]
Let $\psi=\psi(|x|)$ be defined as in \eqref{eq12}, where $\beta>0$ is a constant to be chosen. Set $r\equiv |x|$. We have:
\begin{align}
&\psi'(r)=-\b r(1+r^2)^{-\left(\frac{\b}2+1\right)}\quad \text{for all}\;\; r>0\,, \label{eq320}\\
&\psi''(r)=\b (1+r^2)^{-\left(\frac{\b}{2}+2\right)}[-1 +(\b+1)r^2]\quad \text{for all}\;\; r>0\,. \label{eq321}
\end{align}
At first observe that \eqref{eq42d} is satisfied with the choice $\zeta=\psi$.

\noindent Now, we want to show that $\psi$ solves \eqref{eq40}, for properly chosen $\beta>0$. To do so, we consider separately the cases $(i)-(iv)$.

Suppose that $(i)$ holds.
In view of \eqref{eq320} and \eqref{eq321}, we have:
\begin{equation}\label{eq322}
\begin{split}
&\psi''(r)+\frac{N-2s
+1}{r}\psi'(r)\\
&\quad=\b(1+r^2)^{-\left(\frac{\b}2+2\right)}[(\b-N+2s)r^2-(N-2s+2)]\quad \text{for all}\;\; r>0\,,
\end{split}
\end{equation}
and
\begin{equation}\label{eq322bis}
\begin{split}
&\psi''(r)+\frac{N-1}{r}\psi'(r)=\b(1+r^2)^{-\left(\frac{\b}2+2\right)}[(\b-N+2)r^2-N]\quad \text{for all}\;\; r>0\,.
\end{split}
\end{equation}
Since $0<\b\leq N-2s$, by \eqref{eq322},
\begin{equation}\label{eq323}
\psi''(r)+\frac{N-2s +1}{r}\psi'(r)\leq 0 \quad \text{for all}\;\; r>0\,.
\end{equation}
Now, if $0<\b<N-2$, by \eqref{eq322bis} we also have
\begin{equation}\label{a}
\psi''(r)+\frac{N-1}{r}\psi'(r)\leq 0 \quad \text{for all}\;\; r>0\,.
\end{equation}
Hence, from \eqref{eq323}, \eqref{a} and \eqref{h3},  if $0<\b<N-2$, we obtain, for all $x\in \R^N$,
\begin{equation}\label{eq327}
\begin{aligned}
\LL\psi(x)  -p\,\rho(x)c(x)\psi(x)& \le-  p\,c_0C_0 (1+|x|^2)^{-\frac{\beta}{2}-\frac{\alpha}2}< 0.
\end{aligned}
\end{equation}
On the other hand, if $N-2\le \b<N-2s$, from \eqref{eq323} and \eqref{h3} we obtain, for all $x\in \R^N$,
\begin{equation}\label{eq327bis}
\begin{aligned}
& \LL\psi(x) -p\,\rho(x)c(x)\psi(x) \\
& \qquad \le (1+|x|^2)^{-\frac{\beta}{2}}
\left\{\b(\b-N+2)(1+|x|^2)^{-1}- p\,c_0\,C_0(1+|x|^2)^{-\frac{\alpha}2}\right\}.
\end{aligned}
\end{equation}
(qui \`e dove serve, credo, $\alpha\leq 2$).
Hence, to have that
$$
\LL \psi(x) -p\,\rho(x)c(x)\psi(x)<0,
$$
we require that $$p\,c_0>\frac{\b}{C_0}\left(\b-N+2\right).$$
By \eqref{eq327}, \eqref{eq327bis} and Proposition \ref{prop3}, the conclusion follows, when $(i)$ holds.
\medskip

In order to obtain the thesis of the theorem in the case when 
$\beta>N-2s$, we first note that the following key identity
holds (see the proof of Corollary 4.1 in \cite{FerrVerb}):
\begin{equation}\label{eq328}
-(-\Delta)^s\psi(r)= - \check C F(a,b,c, - r^2)\quad \textrm{for all}\;\; r>1\,,
\end{equation}
where $\check C>0$ is a positive constant, and
$$
a=\frac N 2 +s,\quad b= \frac{\b}2+s, \quad c=\frac N2\,.
$$
By Pfaff's transformation,
\begin{equation}\label{eq329}
F(a,b,c, -r^2)=\frac 1{(1+r^2)^b}F\left(c-a, b, c, \frac{r^2}{1+r^2}\right)\quad \textrm{for all}\,\;r>1\,.
\end{equation}
 Now, if $(ii)$ holds, 
 from Lemma \ref{lemma3}-$(i)$, \eqref{eq328} and \eqref{eq329} we have
\begin{equation}\label{eq330}
-(-\Delta)^s\psi(r)\leq \check C(C_1 +\varepsilon)(1+r^2)^{-\left(s+\frac{\beta}2\right)}\quad \text{whenever}\;\;r>R_\varepsilon\,,
\end{equation}
(for any $\varepsilon>0$ and for some $R_\varepsilon>1$), 
where
$$
C_1=-\frac{\Gamma(\frac N 2)\Gamma\left(\frac{N-\b}2\right)}{\Gamma\left(\frac{N}{2}+s\right)\Gamma\left(\frac{N-\b}{2}-s\right)}>0
$$
(see \eqref{eq319}). Since $\alpha\leq 2s,$ from \eqref{eq330} and due to \eqref{h3}, \eqref{eq27}, we obtain for all $|x|=r>R_\varepsilon$
$$
\begin{aligned}
\LL\psi(x)-p\,\rho(x)c(x)\psi(x) &
 \le (1+|x|^2)^{-\frac{\beta}{2}} \Big\{\check C(C_1+\varepsilon)(1+|x|^2)^{-s}-\beta N(1+|x|^2)^{-1}\\
 &\quad\quad+\beta(\beta+2)|x|^2(1+|x|^2)^{-2}-p\,c_0 C_0(1+|x|^2)^{-\frac{\alpha}2}\Big\}\\
 &\le (1+|x|^2)^{-\frac{\beta}{2}} \Big\{\check C(C_1+\varepsilon)(1+|x|^2)^{-s}\\
 &\quad\quad+\beta(\beta+2)(1+|x|^2)^{-1}-p\,c_0 C_0(1+|x|^2)^{-\frac{\alpha}2}\Big\},
\end{aligned}
$$
Hence,
\begin{equation}\label{eq331}
\LL\psi(x)-p\,\rho(x)c(x)\psi(x)\le0 \quad\textrm{for all}\;\; |x|>R_{\varepsilon}\,,
\end{equation}
if we require that
\begin{equation}\label{eq332}
p\,c_0>\frac 2{C_0}\max\{\check C(C_1+\epsilon)\,;\,\beta(\beta+2)\}\,.
\end{equation}
On the other hand, for all $|x|\leq R_\varepsilon$,
\begin{equation}\label{eq333}
\begin{aligned}
\LL\psi(x) &-p\,\rho(x)c(x)\psi(x)\\
& \le-\left\{-M_{\varepsilon,\beta}-\beta(\beta+2) (1+R_{\varepsilon}^2)^{-1}+p\,c_0\, C_0 (1+R_{\varepsilon}^2)^{-\frac{\beta}{2}-\frac{\alpha}2}\right\}\leq 0
\end{aligned}
\end{equation}
by taking
\begin{equation}\label{eq334}
 p\,c_0>\frac 2{C_0}\left[M_{\epsilon,\b}+\beta(\beta+2) (1+R_{\varepsilon}^2)^{-1}\right](1+R^2_\varepsilon)^{\frac{\beta}2+\frac{\alpha}2},
\end{equation} where
$$
M_{\epsilon,\b}:= \max_{x\in\bar B_{R_\varepsilon}}\big\{\big|-(-\Delta)^s\psi(|x|)\big|\big\}\,.
$$
By \eqref{eq331}, \eqref{eq333} the conclusion follows by Proposition \ref{prop3}, when $(ii)$ holds.

\medskip

Suppose that $(iii)$ holds. From Lemma \ref{lemma3}-$(ii)$ and \eqref{eq329}, for any $\varepsilon>0$, for some $R_\varepsilon>1$, we have:
\begin{equation}\label{eq335}
-(-\Delta)^s\psi(r)\leq \check C(C_2 +\varepsilon)(1+r^2)^{-\left(s+\frac{\b}2\right)}\log(1+r^2)\quad \text{whenever}\;\;|x|>R_\varepsilon\,,
\end{equation}
where
$$C_2=-\frac{\Gamma\left(\frac{\beta}2\right)}{\Gamma(-s)\Gamma\left(\frac{\beta}2+s\right)}>0$$
(see \eqref{eq319}). Since $\alpha<2s$, from \eqref{eq335}, \eqref{eq27} and \eqref{h3}, we obtain for all $|x|>R_\varepsilon$,
\begin{equation}\label{eq336}
\begin{aligned}
\LL\psi(x)&-p\,\rho(x)\,c(x)\psi(x) \\
&\leq (1+|x|^2)^{-\frac{\beta}2}\left\{\check C(C_2+\varepsilon)(1+|x|^2)^{-s}\log(1+|x|^2)\right.-\beta N(1+|x|^2)^{-1}\\
&\quad\quad+\beta(\beta+2)|x|^2(1+|x|^2)^{-2}-p\,c_0 C_0 (1+|x|^2)^{-\frac{\alpha}2}\Big\}\\
 &\le (1+|x|^2)^{-\frac{\beta}{2}} \Big\{\check C(C_2+\varepsilon)(1+|x|^2)^{-s}\log(1+|x|^2)\\
 &\quad\quad+\beta(\beta+2)(1+|x|^2)^{-1}-p\,c_0 C_0(1+|x|^2)^{-\frac{\alpha}2}\Big\}<0,
\end{aligned}
\end{equation}
taking a possibly larger $R_\varepsilon>1$, and requiring that
\begin{equation}\label{eq336b}
p\,c_0>\frac 2{C_0 }\max\left\{\check C(C_2+\varepsilon),\beta(\b+2)\right\}.
\end{equation}
Combining \eqref{eq336} with \eqref{eq333} the conclusion follows, due to Proposition \ref{prop3} also in the case $(iii)$.
\smallskip

 Finally, suppose that $(iv)$ holds.  From Lemma \ref{lemma3}-$(iii)$ and \eqref{eq329}, for any $\varepsilon>0$, for some $R_\varepsilon>1$, we have:
\begin{equation}\label{eq337}
-(-\Delta)^s\psi(r)\leq \check C(C_3 +\varepsilon)(1+r^2)^{-\left(s+\frac N 2\right)}\quad \text{whenever}\;\;r>R_\varepsilon\,,
\end{equation}
where, see \eqref{eq319},
$$C_3=-\frac{\Gamma(\frac N 2)\Gamma\left(\frac{\beta-N}2\right)}{\Gamma(-s)\Gamma\left(\frac{\beta}{2}+s\right)}>0.
$$
Since $\alpha+\beta\leq 2s+N$, from \eqref{eq337}, \eqref{eq27}, and due to \eqref{h3}, we obtain for all $|x|>R_\varepsilon$
\begin{equation}\label{eq338}
\begin{aligned}
\LL \psi(x)-p\,\rho(x)c(x)\psi(x)
& \leq(1+|x|^2)^{-\frac{\beta}2}\left\{\check C(C_3+\varepsilon)(1+|x|^2)^{-s-\frac{N}{2}+\frac{\beta}2}\right.-\beta N(1+|x|^2)^{-1}\\
&\left.\quad\quad+\beta(\beta+2)|x|^2(1+|x|^2)^{-2}-p\,c_0 C_0 (1+|x|^2)^{-\frac{\alpha}2}\right\}\\
&\le (1+|x|^2)^{-\frac{\beta}{2}} \Big\{\check C(C_3+\varepsilon)(1+|x|^2)^{-s-\frac{N}{2}+\frac{\beta}2}\\
 &\quad\quad+\beta(\beta+2)(1+|x|^2)^{-1}-p\,c_0 C_0(1+|x|^2)^{-\frac{\alpha}2}\Big\}<0,
\end{aligned}
\end{equation}
by requiring that
\begin{equation}\label{eq339}
p\,c_0>\frac 2{C_0}\max\{\check C(C_3+\epsilon)\,;\,\beta(\b+2)\}\,.
\end{equation}
On the other hand, \eqref{eq333} holds true, provided \eqref{eq334} is satisfied. In view of \eqref{eq333} and \eqref{eq338}, the conclusion follows by Proposition \ref{prop3}, when $(iv)$ holds. This completes the proof.
\end{proof}

\section{Existence of infinitely many solutions: proofs}\setcounter{equation}{0}\label{mix}
In this section we provide the proof
of  Theorem \ref{cor4}. In order to do this, we exploit some results
from the Weak Theory of the operator $\LL$ which are recalled in Appendix \ref{sec:appendix}.
\medskip

\noindent \textbf{Notation.} Let $x_0\in\R^N$ and $R > 0$ be fixed. In what follows, we denote by
$B_R(x_0)$ the open ball of $\R^N$ with centre $x_0$ and radius $R$, that is,
$$B_R(x_0):=\{x\in\R^N:\,|x-x_0|<R\};$$
in the particular case when $x_0 = 0$, we simply write $B_R$ in place of $B_R(0)$.
\medskip

To show nonuniqueness for problem \eqref{elliptic} we first establish the following
Weak Maximum Principle for \emph{classical} sub/supersolutions of the equation $\LL u + V(x)u = 0$. 
\begin{theorem} \label{thm:WMPclassical}
 Let $\varnothing\neq \Omega\subseteq\R^N$ be a \emph{bounded} open set, and let 
 $V\in C(\Omega),\,\text{$V\geq 0$ in $\Omega$}$. Mo\-re\-o\-ver, let $u\in C^2(\Omega)\cap C(\overline{\Omega})
 \cap \mathcal{L}^s(\R^N)$ be such that
 \begin{equation} \label{eq:conditionsWMPclassical}
 \begin{split}
 	\mathrm{i)}&\,\,\text{$\LL u - V(x)u\leq 0$ in $\Omega$}; \\
 	\mathrm{ii)}&\,\,\text{$u\geq 0$ \emph{pointwise on $\partial\Omega$}}; \\
 	\mathrm{iii)}&\,\,\text{$u\geq 0$ a.e.\,in $\R^N\setminus\overline{\Omega}$}.
 \end{split}
 \end{equation}
 Then, $u\geq 0$ pointwise in $\Omega$.
\end{theorem}
The proof of this result follows the same lines of that of \cite[Theorem 1.3]{BDVV},
where the same result is proved in the case $V\equiv 0$ (and under
the stronger assumption $u\in C(\R^N)$); however, we present it here for the sake of completeness.
\begin{proof}
 We argue by contradiction, assuming that there exists $\xi\in\Omega$ such that
 $u(\xi) < 0$. Since $u$ is continuous on $\overline{\Omega}$, and since
 $u\geq 0$ pointwise on $\partial\Omega$, we can then find $x_0\in \Omega$ such that
 $$u(x_0) = \min_{\overline{\Omega}}u < 0.$$
 This shows that $x_0$ is an \emph{interior minimum point of $u$}, and thus we have
 \medskip
 
 a)\,\,$\Delta u(x_0)\geq 0$ (recall that $u\in C^2(\Omega)$);
 \vspace*{0.1cm}
 
 b)\,\,$V(x_0)u(x_0)\leq 0$ (recall that, by assumption, $V\geq 0$ in $\Omega$).
 \medskip
 
 \noindent This, together with the fact that $\LL u-V(x)u\geq 0$ in $\Omega$, implies
 $$-(-\Delta)^s u(x_0) = \big(\LL u (x_0) - V(x_0)u(x_0)\big) - \Delta u(x_0)+V(x_0)u(x_0) \leq 0.$$
 On the other hand, since $u\geq u(x_0)$ pointwise in $\overline{\Omega}$ (by definition of $x_0$) and since
 $u\geq 0 > u(x_0)$ a.e.\,in $\R^N\setminus\overline{\Omega}$ (see $\mathrm{(iii)}$ 
 in \eqref{eq:conditionsWMPclassical}), 
 by definition of $(-\Delta)^s u$ we have
 $$0\leq (-\Delta)^s u(x_0) = C_{N,s}\,\mathrm{P.V.}\int_{\R^N}
 \frac{u(x_0)-u(y)}{|x-y|^{N+2s}}\,dy \leq 0,$$
 and thus $(-\Delta)^s u(x_0) = 0$. From this, since $u(x_0)-u(y)\leq 0$ for (a.e.) $y\in\R^N$, 
 we then conclude that $u\equiv u(x_0) < 0$ 
 pointwise in $\overline{\Omega}$ and a.e.\,in $\R^N\setminus\Omega$,
 but this is in contradiction with \eqref{eq:conditionsWMPclassical}-$\mathrm{(iii)}$.
\end{proof}

\begin{proposition}\label{prop32}  
Let $\rho, c\in C^\gamma(\mathbb R^N)$ for some $\gamma\in (0,1)$.
We assume that $\rho>0$ in $\R^N$, and that assumption \eqref{eq27} holds.  
If there exist $r_0 > 0$ and a function $V$ such that
\begin{equation} \label{eq26}
\begin{split}
\mathrm{i)}&\,\,\text{$V\in C^2(\R^N\setminus \overline{B}_{r_0})\cap \mathcal{L}^s(\R^N)$ \quad
and \quad $V\in C(\R^N\setminus B_{r_0})$}; \\
\mathrm{ii)}&\,\,\text{$V \geq 0$ in $\R^N\setminus \overline{B}_{r_0}$ \quad and \quad 
$\exists\,\,m_0 > 0$ s.t.\,$V(x)\geq m_0$ for all $x\in\overline{B}_{r_0}$}; \\
\mathrm{iii)}&\,\,\text{$\LL V\leq -\rho$ in $\R^N\setminus B_{r_0}$}; \\
\mathrm{iv)}&\,\,\text{$V(x)\to 0$ as $|x|\to\infty$}.
\end{split}
\end{equation}
\emph{(}for some $r_0 > 0$\emph{)}, 
then there exist infinitely many bounded solutions $u$ of problem \eqref{elliptic}. In particular, for any 
$\eta\in \R$, there exists a solution $u$ to problem \eqref{elliptic} such that
$$
\lim_{|x|\to\infty} u(x)=\eta.
$$
\end{proposition}

\begin{proof}
 Let $\gamma\in\R$ be arbitrarily fixed. For every $n\in\mathbb{N}$, we denote by $u_n$
 the unique weak solution (in the sense of Definition \ref{def:solDirPB}) of the following Dirichlet problem
 $$\mathrm{(D)}_{n,\eta}\qquad\quad \begin{cases}
 \LL u - \rho c(x)u = 0 & \text{in $B_n(0)$}, \\
 u = \eta & \text{in $\R^N\setminus B_n(0)$}.
 \end{cases}$$
 We explicitly stress that the existence and uniqueness of $u_n$ follows from Theorem 
 \ref{thm:existenceDir}, with the choice
 $V(x) = \rho(x)c(x),\,f \equiv 0$ and $g\equiv \eta$: in fact, owing to the assumptions on $\rho,\,c$,
 we have
 \medskip
 
 a)\,\,$V(x) = \rho(x)c(x)\in L^\infty(\R^N)$, and $V\geq 0$ pointwise in $\R^N$;
 \vspace*{0.1cm}
 
 b)\,\,$f\in L^2(\R^N)$ and $g\in H^1_{\mathrm{loc}}(\R^N)\cap \mathcal{L}^s(\R^n)$.
 \medskip
 
 \noindent In particular, since the exterior datum is \emph{constant}, 
 by Remark \ref{rem:assumptionVal} we know that
 \begin{equation} \label{eq:propunVal}
  u_n\in H^1_{\mathrm{loc}}(\R^N)\cap \mathcal{L}^s(\R^N)\qquad\text{and}\qquad
  \iint_{\R^{2N}}\frac{|u_n(x)-u_n(y)|^2}{|x-y|^{N+2s}}\,dx\,dy < \infty.
 \end{equation}
 We now aim at proving that, by possibly passing to a subsequence, the sequence $\{u_n\}_n$
 converges in the $C^2_{\mathrm{loc}}(\R^N)$ topology
 to a function $u\in C^2(\R^N)\cap \mathcal{L}^s(\R^N)$, which satisfies
 $$
   \begin{cases}
   \LL u - \rho(x)c(x)u = 0 & \text{in $\R^N$}, \\
   u(x)\to \eta & \text{as $|x|\to\infty$}.
   \end{cases} 
 $$
 In order to prove this fact, we proceed by steps.
 \medskip
 
 \textsc{Step I:} In this first step we prove that, \emph{for every $n\in\mathbb{N}$}, one has
 \begin{equation} \label{eq:unifBoundun}
  \text{$|u_n|\leq |\eta|$ a.e.\,in $\R^N$}.
 \end{equation}
 To this end, we set $v_n = u_n+|\eta|$ (for a fixed $n\geq 1$) and we preliminary observe that,
 since $u_n$ satisfies \eqref{eq:propunVal}, we
 clearly have $v_n\in H^1_{\mathrm{loc}}(\R^N)\cap \mathcal{L}^s(\R^N)$; moreover,
 since $u_n$ solves $\mathrm{(D)}_{n,\eta}$, it is easy to recognize that
 $v_n$ is a weak {superolution} 
 (in the sense of Definition \ref{def:weaksubsupersol}) of 
 $$\text{$\LL u - \rho(x)c(x)u = 0$ in $B_n$}.$$
 Finally, again by the fact that $u_n$ is a solution of $\mathrm{(D)}_{n,\eta}$ we see that
 \vspace*{0.1cm}
 
 a)\,\,$v_n = \eta+|\eta|\geq 0$ a.e.\,in $\R^N\setminus B_n$;
 \vspace*{0.05cm}
 
 b)\,\,$0\leq v_n^- \leq (u_n-\eta)^-\in H_0^1(B_n)$, and hence $v_n^-\in H_0^1(B_n)$.
 \vspace*{0.1cm}
 
 \noindent Gathering all these facts, we are entitled to apply the Weak Maximum Principle in
 Theorem \ref{thm:wmpsubsol}, thus obtaining $v_n\geq 0$ a.e.\,in $B_n$, that is,
 $u_n\geq -|\eta|$ a.e.\,in $B_n$ (hence, a.e.\,in $\R^N$).
 
 On the other hand, setting $w_n = -u_n+|\eta|$ and arguing \emph{exactly} as above (notice that
 $-u_n$ is a weak solution of $\mathrm{(D)}_{n,-\eta}$), we derive that
 $w_n\geq 0$ a.e.\,in $B_n$, that is, $u_n\leq |\eta|$ a.e.\,in $B_n$ (hence, a.e.\,in $\R^N$).
 Summing up, we conclude that
 $$\text{$-|\eta|\leq u_n\leq |\eta|$ a.e.\,in $\R^N$},$$
 and this is exactly the claimed \eqref{eq:unifBoundun}.
 \medskip
 
 \textsc{Step II:} In this second step we prove the following \emph{$C^{2,\gamma}_{\mathrm{loc}}$-estimate}
 for the functions $u_n$: for every fixed bounded open set $\mathcal{O}\subseteq\R^N$
 there exists $\vartheta > 0$, only depending on 
 $N,\,s,\,\gamma,\,\eta,\,\|c\|_{C^\gamma(\R^N)}$ and on the open set
 $\mathcal{O}$ (hence, independent of $n$), such that
 \begin{equation} \label{eq:C2alfaVal}
  \|u_n\|_{C^{2,\alpha}(\overline{\mathcal{O}})} \leq \vartheta\quad\text{for every $n\in\mathbb{N}$ with
  $O\Subset B_n$}.
 \end{equation}
 To this end, we arbitrarily fix a number 
 $n\in\mathbb{N}$ such that $\mathcal{O}\Subset B_n$
 and we observe that, on ac\-count of \eqref{eq:unifBoundun}, the function $u_n$
 is a \emph{bounded solution} of the {semilinear equation}
 $$
 \text{$-\LL u = g(x,u)$ in $B_n$},
 $$
 where the non-linearity $g:\R^N\times\R\to\R$ is given by
 $$g(x,t) = -\rho(x)c(x)\min\{|\eta|,t\}.$$
 Now, since we are assuming $c\in C^\gamma(\R^N)$, we clearly have 
 $g\in C^\gamma(\R^N\times\R)$; as a consequence, 
 since $u_n$ satisfies \eqref{eq:propunVal}, we can apply
 the \emph{interior $C^{2,\gamma}$-estimate}
 in \cite[Theorem 1.6]{SVWZ}: there exists a constant $C > 0$, only depending on
 $N,s,\gamma$ and $\mathcal{O}$, such that
 \begin{align*}
  \|u_n\|_{C^{2,\alpha}(\overline{\mathcal{O}}} & \leq
  C\big(\|u_n\|_{L^\infty(\R^N)}+\|g\|_{C^\gamma(\overline{O}_{\frac{7\rho}{8}}\times\R}\big)
  \big(1+\|g\|_{C^\gamma(\overline{O}_{\frac{7\rho}{8}}\times\R}\big) \\
  & \leq C(|\eta|+\|c\|_{C^\gamma(\R^N)})(1+\|c\|_{C^\gamma(\R^N)}) =:\vartheta
 \end{align*}
 (here, $\rho = \mathrm{dist}(\mathcal{O},\partial B_n)$ and, for every $\delta > 0$,
 we set $\overline{\mathcal{O}}_\delta = \{x\in\R^N:\,\mathrm{dist}(x,\overline{\mathcal{O}} < \delta\}$), 
 and this is exactly the claimed \eqref{eq:C2alfaVal} (since $\vartheta$ only depends
 on $N,s,\gamma,\eta,\|c\|_{C^\alpha(\R^N)}$ and on $\mathcal{O}$).
 \vspace*{0.05cm}
 
 We explicitly observe that, on account of \eqref{eq:C2alfaVal}, we have 
 $u_n\in C^{2,\gamma}_{\mathrm{loc}}(B_n)$; in particular, since $u_n$ is a 
 (weak) solution
 of $\mathrm{(D)}_{n,\eta}$,
 a standard integration-by-parts argument shows that 
 \begin{equation} \label{eq:unclassical}
  \text{$\LL u_n - \rho(x)c(x)u_n = 0$ pointwise in $B_n$}.
 \end{equation}
 
 \textsc{Step III:} In this third step we prove the \emph{global continuity} of $u_n$ (for all $n\geq 1$).
 To this end, we use once again the regularity results established in \cite{SVWZ} for
 \emph{semilinear equations driven by $\LL$}.
 \vspace*{0.05cm}
 
 First of all, since $u_n$ solves $\mathrm{(D)}_{n,\eta}$ and it satisfies 
 \eqref{eq:propunVal}-\eqref{eq:unifBoundun},
 it is straightforward to recognize that the function 
 $v_n = u_n-\eta$ is a weak solution (in the sense of \cite[Definition 2.1]{SVWZ}) of
 $$\begin{cases}
  -\LL u = g(x,u) & \text{in $B_n$}, \\
  u = 0 & \text{in $\R^N\setminus B_n$}
 \end{cases}
 $$
 where $g$ is as in \textsc{Step II}; thus, since $g\in C^\gamma(\R^N\times\R)$ 
 (hence, in particular,
 $g$ is \emph{globally bounded in $\R^N\times\R$}) 
 and since, obviously, $\partial B_n$ is of class $C^{1,1}$,
 we are entitled to apply the \emph{global $C^{1,\beta}$-e\-sti\-ma\-tes} in \cite[Theorem 1.3]{SVWZ},
 ensuring that
 \begin{equation} \label{eq:vnC1beta}
  v_n\in C^{1,\beta}(\overline{B}_n)\quad\text{for all $\beta\in (0,\min\{1,2-2s\})$}.
 \end{equation}
 Now, since $v_n|_{B_n}\in H_0^1(B_n)$, from \eqref{eq:vnC1beta} we infer that
 $v_n\equiv 0$ on $\partial B_n$; this, together with the fact that $v_n\equiv 0$ almost everywhere
 in $\R^N\setminus
 B_n$, ensures that $v_n\in C(\R^N)$ (up to possibly modify $v_n$ on a set of zero Lebesgue measure).
 Finally, recalling that
 $$u_n = v_n+\eta,$$
 we conclude that $u_n\in C^{1,\beta}(\overline{B}_n)\cap C(\R^N)$ for every $\beta\in 
 0<\beta<\min\{1,2-2s\}$. In particular,
 we deduce that \eqref{eq:unifBoundun} holds \emph{pointwise on $\R^N$}, that is,
 \begin{equation} \label{eq:unifBoundunPointwise}
  \text{$|u_n(x)|\leq \eta$ pointwise for every $x\in\R^N$}.
 \end{equation}
 
 \textsc{Step IV:} In this last step we complete the proof of the proposition. To begin with we
 observe that, on account \eqref{eq:C2alfaVal}, we can perform
 a diagonal-type based on the Arzel\`a-As\-co\-li Theorem: this provides us
 with a (unique) function $u_\eta\in C^2(\R^N)$ such that (up to a subsequence)
 \begin{equation} \label{eq:convunugamma}
 \begin{gathered}
  \text{$D^{\alpha}u_n \xrightarrow[n\to\infty]{} 
  D^\alpha u_\gamma$ \emph{uniformly on the compact subsets of $\R^N$}} \\
  \text{for every multi-index $\alpha$ with $|\alpha|\leq 2$}.
  \end{gathered}
 \end{equation}
 This, together with \eqref{eq:unifBoundunPointwise}, easily implies that $|u_\eta|\leq |\eta|$ in $\R^N$,
 and hence $u_\eta\in \mathcal{L}^s(\R^N)$. Gathering all these facts,
 and taking into account \eqref{eq:unclassical}, we then infer that
 \begin{equation} \label{eq:ugammasolRN}
  \text{$\LL u_\eta-\rho(x)c(x)u_\eta = 0$ pointwise in $\R^N$}
  \end{equation}
 (we explicitly stress that, in order to compute $(-\Delta)^s u_\eta$, one needs
 to exploit a simple dominated-con\-ver\-gen\-ce argument based on
 \eqref{eq:convunugamma}). With \eqref{eq:ugammasolRN} at hand,
 to complete the demonstration we now turn to show that $u_\eta$ has the desired
 behaviour at infinity, that is,
 \begin{equation} \label{eq:toprovebehaviorugamma}
  \lim_{|x|\to\infty}u_\eta = \eta.
 \end{equation}
 To this end, if $V$ is as in the statement, we set
 $\overline{w} = C_1V+\eta$ (for a suitable constant $C_1 > 0$ to be chosen in a moment).
 Taking into account \eqref{eq27} and \eqref{eq26}, for all $x\in\R^N\setminus B_{r_0}$ we have
 \begin{align*}
  \LL \overline{w}(x)-\rho(x)c(x)\overline{w}(x) & \leq -C_1\rho(x)-\rho(x)c(x)\gamma
   \leq -\rho(x)(C_1+\gamma c_0) < 0,
 \end{align*}
 provided that $C > 0$ is sufficiently large (recall that, by assumption, $\rho > 0$ in $\R^N$); 
 in addition, since $V$ satisfies property $\mathrm{ii)}$ in \eqref{eq26}, we derive that
\begin{equation} \label{eq:estimoverwgeq}
 \text{$\overline{w} \geq \eta$ in $\R^N\setminus \overline{B}_{r_0}$}\qquad\text{and}\qquad
\text{$\overline{w}\geq C_1m_0+\eta\geq |\eta|$ in $\overline{B}_{r_0}$},
\end{equation}
by enlarging $C_1$ if needed.
In view of these facts, if we define
 $$\Omega := B_n\setminus \overline{B}_{r_0}\qquad \text{and}\qquad v_n = \overline{w}-u_n$$
 (for all $n\in\mathbb{N}$ with $n > r_0$), we easily see that
 \vspace*{0.1cm}
 
 a)\,\,$v_n\in C^2(\Omega)\cap \mathcal{L}^s(\R^N)$ and $v_n\in C(\overline{\Omega})$ (see
 \eqref{eq26} and recall that $u_n\in C(\R^N)$);
 \vspace*{0.1cm}

 b)\,\,$\LL v_n-\rho(x)c(x)v_n \geq 0$ pointwise in $B_n\setminus \overline{B}_{r_0}$;
 \vspace*{0.1cm}
 
 c)\,\,$v_n\geq 0$ pointwise in $\R^N\setminus\Omega = \overline{B}_{r_0}\cup (\R^N\setminus B_n)$
 (see \eqref{eq:unifBoundunPointwise} and
 \eqref{eq:estimoverwgeq}).
  \vspace*{0.1cm}
  
  \noindent We are then entitled to apply the Weak Maximum Principle
  for \emph{classical} supersolutions in Theorem \ref{thm:WMPclassical}, 
  thus obtaining $v_n\geq 0$ in $\Omega = B_n\setminus
  \overline{B}_{r_0}$ (hence, in $\R^N\setminus
  \overline{B}_{r_0})$), that is,
  \begin{equation}  \label{eq:unleqoverw}
   \text{$u_n\leq C_1V+\eta$ pointwise in $\R^N\setminus \overline{B}_{r_0}$}.
   \end{equation}
  On the other hand, setting $\underline{w} := -C_2V+\gamma$ 
  and arguing \emph{exactly as above},
  we see that it is possible to choose
  the constant $C_2 > 0$ (independently of $n > r_0$) in such a way that
  \begin{equation}  \label{eq:ungequnderw}
   \text{$u_n\geq -C_2V+\eta$ pointwise in $\R^N\setminus \overline{B}_{r_0}$}.
   \end{equation}
   Summing up, by combining \eqref{eq:unleqoverw}-\eqref{eq:ungequnderw} we obtain
  \begin{equation} \label{eq:unbarriertoconclude}
    \text{$-C_2V+\eta\leq u_n \leq  C_1V+\eta$ pointwise in $\R^N\setminus \overline{B}_{r_0}$},
\end{equation}     
   and this estimate holds \emph{for every $n > r_0$} (with constants $C_1,C_2 > 0$ independent of $n$).
   
   With \eqref{eq:unbarriertoconclude}
   at hand, we are ready
   to conclude the proof of \eqref{eq:toprovebehaviorugamma}: in fact, letting $n\to\infty$ in the above
   \eqref{eq:unbarriertoconclude} (and recalling that $u_n\to u_\eta$ locally uniformly in $\R^N$), we get
   $$
 	\text{$-C_2V+\eta\leq u_\eta \leq  C_1V+\eta$ pointwise in $\R^N\setminus \overline{B}_{r_0}$};
   $$
   from this, letting $|x|\to\infty$ and using property $\mathrm{iv)}$ in \eqref{eq26}, we
   obtain
   \eqref{eq:toprovebehaviorugamma}. 
\end{proof}

We are now left to show that a supersolution $V$ to problem \eqref{eq26} exists. This is the content of the next

\begin{lemma}\label{lemma5}
Let $\rho\in C(\R^N)$ be such that
\begin{equation} \label{eq:reverseHzerorho}
 0< \rho\leq \frac{c_0}{(1+|x|^{2})^{\alpha/2}}\quad\text{for every $x\in\R^N,\,|x|\geq r_0$},
\end{equation}
\emph{(}for some constants $c_0,\,r_0 > 0$\emph{)}, and assume that
\begin{equation}\label{eq51}
\alpha>2s.
\end{equation}
Then there exists a function $V$ satisfying properties $\mathrm{i)-iv)}$ in Proposition \ref{prop32}.
\end{lemma}

\begin{proof}
 We arbitrarily fix $C > 0$ (to be appropriately chosen in a moment), and 
 \begin{equation}\label{beta}
0<\beta<\min\{N-2, \alpha-2s\};
\end{equation}
accordingly, we define the following function
$$V:\R^N\to\R,\qquad V(x) = \begin{cases}
C|x|^{-\beta} & \text{if $x\neq 0$}, \\
0 & \text{if $x = 0$.}
\end{cases}$$
We explicitly stress that the definition of $V$ for $x = 0$ is totally immaterial, since it does
not play any role in the computation of $\LL V$ on the open set $\R^N\setminus\{0\}$.
\vspace*{0.05cm}

We now observe that, by definition, this function $V$ obviously
satisfies properties $\mathrm{ii)}$ and $\mathrm{iv)}$ in the above \eqref{eq26}
(for every choice of $r_0 > 0$);
moreover
since $\beta < N$, it is easy to see that 
$$V\in C^2(\R^N\setminus\{0\})\cap \mathcal{L}^s(\R^N)$$
(hence, also property $\mathrm{i)}$ is satisfied for every $r_0 > 0$), and 
we can
 compute $\LL V$ pointwise for $x\neq 0$,
 see  Pro\-po\-si\-tion \ref{prop:welldefDeltas}.
 In particular, since it is quite standard to recognize that
 \begin{equation} \label{eq:DeltasV}
  (-\Delta)^s V(x) = \vartheta|x|^{-\beta-2s}\quad\text{for every $x\neq 0$},
 \end{equation}
 (where $\vartheta > 0$ is a suitable constant only depending on $N,s,\beta$, see, e.g.,
 \cite{Kw}), we get
 \begin{equation}\label{eq342}
\begin{aligned}
\mathcal{L} V(x) & = C\big[\beta(\beta+2-N)|x|^{-\beta-2} -\vartheta|x|^{-\beta-2s}\big]\\
&\leq -C\vartheta|x|^{-\beta-2s} 
\leq -C\vartheta|x|^{-\alpha}\quad\text{for all $x\in\R^N$ with $|x|\geq 1$},
\end{aligned}
\end{equation}
where in the last inequality we have used the fact that $\beta+2s<\alpha$, see \eqref{beta}.
On the other hand, since the function 
$\rho$ satisfies assumption \eqref{eq:reverseHzerorho}, we have
\begin{equation} \label{eq:rholeqdacombinare}
 \rho(x)  \leq \frac{c_0}{(1+|x|^2)^{\alpha/2}} \leq c_0|x|^{-\alpha}\quad
 \text{for every $x\in\R^N$ with $|x| > r_0$}.
\end{equation}
Gathering \eqref{eq342}-\eqref{eq:rholeqdacombinare}, we then obtain
$$\LL V \leq -\frac{C\vartheta}{c_0}\rho\leq -\rho\quad\text{for every $x\in\R^N$ with $|x|>r_0$},$$
provided that $C > 0$ is sufficiently large.
This shows that $V$ also satisfies property $\mathrm{iii)}$ in \eqref{eq26}
(with $r_0 > 0$ as in \eqref{eq:reverseHzerorho}),
and the proof is complete.
\end{proof}

\begin{proof}[Proof of Theorem \ref{cor4}]
By combining Proposition \ref{prop32} and Lemma \ref{lemma5}, the thesis follows.
\end{proof}

\section{Further results: parabolic equations}\setcounter{equation}{0} \label{pp}

By minor modification of the proof of the uniqueness result of Theorem \ref{teo3}, we are able to show uniqueness of solutions to the linear local-nonlocal Cauchy problem:
\begin{equation}\label{problema}
\begin{cases}
\rho\, u_t- \Delta u+(-\Delta)^s u   =0 \quad & \text{in}\,\,\,S_T:=\R^N\times(0,T]\\
u=0\quad & \text{in}\,\,\,\R^N\times\{0\}.
\end{cases}
\end{equation}

\subsection{Preliminaries and main results}

We deal with solution in the sense of the following definitions.
\begin{definition}\label{defsoleqp}
We say that a function $u$ is a {\em solution} to equation
\begin{equation}\label{eq24}
\rho\, u_t- \Delta u +(-\Delta)^s u =\,0\quad \textrm{in}\;\; S_T\,,
\end{equation}
if
\begin{itemize}
\item[(i)]  $u\in C^2(S_T)$, for each $t\in (0,T]$  $u(\cdot, t)\in \mathcal L^s(\R^N)$;
\item[(ii)] $ \rho(x) u_t \,- \Delta u+ C_{N,s} \textrm{P.V}.\, \displaystyle\int_{\R^N} \frac{u(x,t)-u(y,t)}{|x-y|^{N+2s}}d y =\,0$\;\;  for all\;\; $(x,t)\in
S_T$\,.
\end{itemize}
Furthermore, we say that $u$ is a {\em supersolution\; (subsolution)} to equation \eqref{eq24}, if in $(ii)$ instead of $``="$ we have $``\geq"\; (``\leq")$\,.
\end{definition}

\begin{definition}\label{defsolp}
We say that a function $u$ is a solution to problem \eqref{problema} if
\begin{itemize}
\item[(i)] $u\in C^2(\bar S_T), u\in L ^1\big((0,T),\mathcal L^s(\R^N)\big)$\,;
\item[(ii)]  $u$ is a solution to equation \eqref{eq24}\, in the sense of Definition \ref{defsoleqp};
\item[(iii)] $u(x,0)=0$\;\; for all\;\; $x\in \R^N\,.$
\end{itemize}
\end{definition}

We can state the uniqueness result for problem \eqref{problema} in the following
\begin{theorem}\label{teo1}
Let assumptions \eqref{h3} be satisfied. Let $u$ be a solution to problem \eqref{problema} with $|u(\cdot, t)|^p\in \mathcal L^s(\R^N)$, for some
$p\geq 1$, for each $t>0$. Assume that one of the conditions $(i)-(iv)$ of Theorem \ref{teo3} holds. Let $\psi$ be defined as in \eqref{eq12}. If $u\in L^p_{\psi}(S_T)$, then
$$
u\equiv 0 \quad \textrm{in} \;\; S_T.
$$
\end{theorem}
Moreover, analogously to Corollary \ref{cor3}, we have the following
\begin{corollary}\label{cor1}
Let assumption $(H_0)$ be satisfied. Let $u$ be a solution to problem \eqref{problema}. Suppose that $\alpha<2s$. If
$$|u(x,t)| \leq C(1+|x|^2)^{\frac{\vartheta}2}\quad \textrm{for all}\;\; x\in S_T,$$
for some $C>0$ and $0\le \vartheta<2s-\alpha$, then
$$u\equiv 0\quad \textrm{in}\;\; S_T\,.$$
\end{corollary}

In order to prove Corollary \ref{cor1} it suffices to apply Theorem \ref{teo1} with $\beta=N+2s-\alpha>N$ and $p=1$.

\subsection{Proof of Theorem \ref{teo1}}

For any $R>0$ let
\begin{equation}\label{eq32}
\gamma_R(x):= \gamma\left(\frac{|x|}{R}\right) \quad \textrm{for all}\;\; x\in \R^N\,;
\end{equation}
with $\gamma$ defined as in \eqref{eq31}, and for any $\tau\in (0,T)$ let
$$
S_\tau:= \R^N\times (0,\tau]\,.
$$

We start by proving a general criterion for uniqueness of nonnegative solutions to problem
\eqref{problema} in $L^1_{\psi}(S_T),$ where $\psi$ is defined as in \eqref{eq12} for some constant $\beta>0$.

\begin{proposition}\label{prop1}
Let assumptions \eqref{h3} be satisfied. Let $u$ be a solution to problem \eqref{problema} with $|u(\cdot, t)|^p\in \mathcal L^s(\R^N)$ for some $p\geq 1$, for each $t>0$. Assume that there exists a positive supersolution $\phi\in C^2(\bar S_T)$ to equation
\begin{equation}\label{eq32b}
\rho \phi_t+\Delta \phi-(-\Delta)^s\phi   =0 \quad \textrm{in}\;\; S_T\,,
\end{equation}
such that
\begin{equation}\label{eq32d}
\phi(x,t)+|\nabla \phi(x,t)|\leq C \psi(x)\quad \textrm{for all}\;\; (x,t)\in S_T,
\end{equation}
for some constant $C>0$. If $u\in L^p_\psi(S_T)$, then
$$u\equiv 0\quad \textrm{in}\;\; S_T\,.$$
\end{proposition}


To prove Proposition \ref{prop1}, we need an auxiliary lemma which is the parabolic analogue of Lemma \ref{lemma4}.

\begin{lemma}\label{lemma1}
Let $\tau\in (0,T)$, $\phi\in C^2(\bar S_\tau)$, $\phi>0$; suppose that \eqref{eq32d} is satisfied. Let $v\in L^1_{\psi}(S_\t)$. Then
\begin{equation}\label{eq31bis}
\int_0^\t\int_{\R^N}|v(x,t)|\phi(x,t)|(-\Delta)^s\g_R(x)|\,dxdt +\int_0^\t\int_{\R^N}|v(x,t)|\,|\mathcal B(\phi, \g_R)(x)|dx dt\longrightarrow 0 \end{equation} and
\begin{equation}\label{eq31tris}
\int_0^\t\int_{\R^N}|v(x,t)|\phi(x,t)|\Delta\g_R(x)|\,dxdt \,+\int_0^\t\int_{\R^N}|v(x,t)|\,|\left\langle \nabla\phi(x,t), \nabla\g_R(x)\right\rangle| dx dt\longrightarrow 0,
\end{equation}
as $R\to\infty$.
\end{lemma}

\medskip

\begin{proof}[ Proof of Proposition \ref{prop1}] Let $\t\in(0,T)$. Take a nonnegative function
$$v\in C^2(\bar S_\t)\,\, \text{with}\,\,\operatorname{supp}v(\cdot, t)\,\, \text{compact for each}\,\,t\in [0,\t].$$
Moreover, take a function $w\in C(\bar S_\t)\cap  L^1\big((0,\t);\mathcal L^s(\R^N)\big)$ such that $w_t\in C(S_T)$ and for each $t\in(0,\t]$,
$$
w(\cdot, t)\in \mathcal L^s(\R^N)\cap C^2(\R^N).
$$
For any $\epsilon\in (0,\tau)$, integrating by parts we have:
\begin{equation}\label{eq36}
\begin{aligned}
& \int_\epsilon^\t\int_{\R^N} v \big[ \Delta w-(-\Delta)^s w-  \rho w_t  \big]\,  dx dt \\
&= \int_ \epsilon ^\t \int_{\R^N} w\big[ \Delta v- (- \Delta)^s v+  \rho v_t \big]\, dx dt\\
&-\int_{\R^N} \rho(x) v(x,\tau)w(x,\tau) dx +\int_{\R^N} \rho(x) v(x,\epsilon) w(x,\epsilon) dx\,.
\end{aligned}
\end{equation}
Let $G_\alpha$ be defined as in \eqref{eq37}. By the differential equation in problem \eqref{problema} and due to \eqref{eq37}, \eqref{eq38},
\begin{equation}\label{eq39}
\rho [G_\alpha(u)]_t=\rho G_\alpha'(u) u_t=-G_\alpha'(u)[-\Delta u +(-\Delta)^s u]\quad \textrm{in}\;\, S_T\,.
\end{equation}
From \eqref{eq38}, \eqref{eq39bis}, \eqref{eq39} and Lemma \ref{lemma2}, since $p\ge 1$, we obtain
\begin{equation}\label{eq310}
\rho [G_\alpha(u)]_t -\Delta [G_{\alpha}(u)]+ (-\Delta)^s[G_\alpha(u)]\le -G_{\alpha}''(u)|\nabla u|^2\le 0 \quad \textrm{in}\;\; S_T\,.
\end{equation}
So, from \eqref{eq36} with $w=G_\alpha(u)$ and \eqref{eq310} we obtain
\begin{equation}\label{eq311}
\begin{aligned}
\int_{\R^N} \rho(x) G_\alpha[u(x,\tau)] v(x,\tau) \,dx &\leq  \int_ \epsilon ^\tau\int_{\R^N} G_\alpha(u) \big[\Delta v- (- \Delta)^s v+\rho v_t \big]\, dx dt \\
& + \int_{\R^N} \rho(x) v(x,\epsilon) G_\alpha[u(x,\epsilon)] \,dx\,.
\end{aligned}
\end{equation}
Letting $\epsilon\to 0^+$ in \eqref{eq311}, by the dominated convergence theorem,
\begin{equation}\label{eq312}
\begin{aligned}
\int_{\R^N} \rho(x) G_\alpha[u(x,\t)] v(x,\tau)\, dx &\leq  \int_0^\tau \int_{\R^N} G_\alpha(u) \big[\Delta v- (- \Delta)^s v+\rho  v_t \big] \,dx dt \\
&\quad + \alpha^{p/2} \int_{\R^N} \rho(x) v(x,0)\, dx\,.
\end{aligned}
\end{equation}
Now, letting $\alpha\to 0^+$ in \eqref{eq312}, by the dominated convergence theorem,
\begin{equation}\label{eq313}
\int_{\R^N} \rho(x)  |u(x,\t)|^p v(x,\tau) \, dx \leq  \int_0^\tau \int_{\R^N} |u|^p \big[\Delta v- (- \Delta)^s v+ \rho  v_t  \big] \,dx dt
\end{equation}

For any $R>0$, we can choose
$$
v(x,t):= \phi(x,t)\gamma_R(x)\quad \textrm{for all}\;\; (x, t)\in \bar S_\t\,.
$$
Using the fact that $\phi$ is a supersolution to equation \eqref{eq32b} and $\gamma_R\geq 0$, we obtain
\begin{equation}\label{eq314}
\begin{aligned}
 \rho v_t+\Delta v- (- \Delta)^s v&=\gamma_R \left[ \rho  \phi_t +\Delta\phi-(-\Delta)^s \phi  \right]\\
&\quad\,+\phi\Delta\gamma_R + 2\left\langle \nabla\phi, \nabla \gamma_R\right \rangle- \phi(-\Delta)^s\gamma_R + \mathcal B(\phi,\g_R)  \\
& \leq \phi\Delta\gamma_R + 2\,|\left\langle \nabla\phi, \nabla \gamma_R\right \rangle|- \phi(-\Delta)^s\gamma_R + \mathcal B(\phi,\g_R)  \quad\quad \text{in}\;\; S_\tau\,.
\end{aligned}
\end{equation}
Since $|u|^p\geq 0$, by \eqref{eq313} and \eqref{eq314} we conclude
that
\begin{equation}\label{eq315}
\begin{aligned}
\int_{\R^N} \rho(x) & |u(x,\tau)|^p \phi(x,\tau)\gamma_R(x) dx \\
&\,\,\leq \int_0^\t \int_{\R^N} |u|^p \big[ \phi\,\Delta\gamma_R + 2\,|\left\langle \nabla\phi, \nabla \gamma_R\right \rangle|- \phi(-\Delta)^s\gamma_R + \mathcal B(\phi,\gamma_R)  \big]dx dt.
\end{aligned}
\end{equation}
Finally, from Lemma \ref{lemma1} with $v=|u|^p$ and the monotone convergence theorem, sending $R\to \infty$ in \eqref{eq315} we get
\begin{equation}\label{eq316}
\int_{\R^N} \rho(x) |u(x,\t)|^p\phi(x,\t)dx\leq 0\,.
\end{equation}
From \eqref{eq316} and \eqref{h3}, since $\phi>0$ in $S_\tau$ and $|u|^p\geq 0$ we infer that $u\equiv 0$ in $S_\tau$. This completes
the proof.
\end{proof}

\medskip


\begin{proof}[Proof of Theorem \ref{teo1}.]
Let $\psi=\psi(|x|)$ be defined as in \eqref{eq12}. Then, \eqref{eq320} and \eqref{eq321} are in force. For any $\lambda>0$ define
$$
\phi(x,t):= e^{-\l t}\psi(r)\quad \text{for all}\;\; (x,t)\in \bar S_T\,.
$$
At first observe that \eqref{eq32d} is satisfied. Indeed, one has, for some $C>0$,
$$
\begin{aligned}
\phi(x,t)+|\nabla\phi(x,t)|&\le\, e^{-\lambda t}\left\{\psi(r)+\psi'(r)\right\}\le C\,\psi(r) \quad\quad \text{for all}\,\,(x,t)\in \bar S_T.
\end{aligned}
$$

Now, we want to show that $\phi$ is a supersolution to equation \eqref{eq32b}.  Observe that $\lambda$ satisfies the same conditions as $pc_0$ in the proof of Theorem \ref{teo3}. To be specific, let us consider all the cases $(i)-(iv)$ as well as in the proof of Theorem \ref{teo3}.

\smallskip

\begin{itemize}
\item Let $(i)$ holds. We require that $$\lambda>\frac{1}{C_0}\left(\b-N+2\right).$$
\item Let $(ii)$ hold. We require that, for some $\varepsilon>0$,
$$
\lambda  >\frac2{C_0}\max\left\{\check C(C_1+\epsilon)\,;\,\beta(\beta+2),\, \left[M_{\epsilon,\b}+\beta(\beta+2) (1+R_{\varepsilon}^2)^{-1}\right](1+R^2_\varepsilon)^{\frac{\beta}2+\frac{\alpha}2}\right\}\,,
$$
(see \eqref{eq332}, \eqref{eq334}).
\item Let $(iii)$ hold. It is sufficiently to have, for some $\varepsilon>0$
$$
\lambda  >\frac 2{C_0}\max\left\{\check C(C_2+\varepsilon),\, \beta(\beta+2),\, \left[M_{\epsilon,\b}+\beta(\beta+2) (1+R_{\varepsilon}^2)^{-1}\right](1+R^2_\varepsilon)^{\frac{\beta}2+\frac{\alpha}2}\right\}\,,
$$
(see \eqref{eq336b}, \eqref{eq334}).
\item Finally, Let $(iv)$ hold. We ask that, for some $\varepsilon>0$
$$
\lambda >\frac 2{C_0}\max\left\{\check C(C_3+\varepsilon),\, \beta (\beta+2),\, \left[M_{\epsilon,\b}+\beta(\beta+2) (1+R_{\varepsilon}^2)^{-1}\right](1+R^2_\varepsilon)^{\frac{\beta}2+\frac{\alpha}2}\right\}\,,
$$
(see \eqref{eq339}, \eqref{eq334}).
\end{itemize}

Thus, by Proposition \ref{prop1} the conclusion follows.

\end{proof}

\appendix

\section{Maximum principle and the Dirichlet problem for $\LL-V(x)$} \label{sec:appendix}
In order to make the paper as self-contained as possible,
we collect in this Appendix the relevant definitions and results
on the mixed operator
$\LL_V = \LL u - V(x)u $
which have been used in the proof of Theorem \ref{cor4}.
To be more precise, we will be concerned with the following topics:
\begin{itemize}
 \item[1)] the weak maximum principle for weak sub/supersolutions;
 \item[2)] existence and (interior) regularity of the weak solutions of the $\LL_V$-Dirichlet problem.
\end{itemize}
\subsection*{1) The weak maximum principle for $\LL$.}
In this first part of the section we prove a weak ma\-xi\-mum/comparison principle
for \emph{weak sub/supersolution} of the equation $\LL u + V(x)u = 0$. In order to
introduce the adequate 
\emph{functional setting} to develop a \emph{weak theory} for the operator $\LL$,
we first need to extend the fractional Laplacian to \emph{non-regular}
 functions.
\medskip

\noindent\textbf{1.1) The weak fractional Laplacian.}
Let $\Omega\subseteq\R^N$ be an arbitrary open, bounded or not.
As the classical Laplacian (formally corresponding to the choice $s = 1$) is related,
via the Euler-Largange euqation,
to the Sobolev space $H^1(\Omega)$, the \emph{fractional Laplacian} is related
to the \emph{fractional Sobolev space} $H^s(\Omega)$, which is defined as follows:
$$H^s(\Omega) := \Big\{u\in L^2(\Omega):\,[u]^2_{s,\Omega} = \iint_{\Omega\times\Omega}
\frac{|u(x)-u(y)|^2}{|x-y|^{N+2s}}\,dx\,dy < \infty\Big\}.$$
While we refer to the recent monograph \cite{LeoniFract} for a thorough introduction
on fractional Sobolev spaces, here we list the few basic properties of $H^s(\Omega)$
we will exploit in this paper.
\begin{itemize}
 \item[i)] $H^s(\Omega)$ is endowed with a structure of a real Hilbert space by the scalar product
 $$\langle u,v\rangle_{s,\,\Omega} 
 := \iint_{\Omega\times\Omega}\frac{(u(x)-u(y))(v(x)-v(y))}{|x-y|^{N+2s}}\,dx\,dy\qquad
 (u,v\in H^s(\Omega)).$$ 
 
 \item[ii)] $C_0^\infty(\Omega)$ is a \emph{linear subspace of $H^s(\Omega)$}; in addition,
 in the particular case when $\Omega = \R^N$, it can be proved that
  $C_0^\infty(\R^N)$ is a \emph{dense subspace} of $H^s(\R^N)$.
 
 \item[iii)] If $\Omega = \R^N$ or if $\Omega$ has \emph{bounded boundary $\partial\Omega\in C^{0,1}$},
 the Sobolev space $H^1(\Omega)$ is \emph{continuously embedded into} $H^s(\Omega)$,
 that is, there exists $\mathbf{c} = \mathbf{c}(N,s) > 0$ such that
 \begin{equation} \label{eq:H1embeddingHs}
  \iint_{\Omega\times\Omega}\frac{|u(x)-u(y)|^2}{|x-y|^{N+2s}}\,dx\,dy \leq 
  \mathbf{c}\,\|u\|_{H^1(\Omega)}^2\quad\text{for every $u\in H^1(\Omega)$}.
 \end{equation}
  In particular, if $\Omega\subseteq\R^N$ is a
  \emph{bounded open set} (with no regularity assumptions on $\partial\Omega$) and if
  we denote by $\hat{u}\in H^1(\R^N)$ the \emph{zero-extension} of a function $u\in H^1(\Omega)$
  (that is, we define $\hat{u} = u\cdot\mathbf{1}_\Omega$),
  we can find a constant $\gamma = \gamma(N,s,\Omega) > 0$ such that
  \begin{equation} \label{eq:H01embeddingHs}
  \iint_{\R^{2N}}\frac{|\hat{u}(x)-\hat{u}(y)|^2}{|x-y|^{N+2s}}\,dx\,dy \leq 
  \gamma\,\int_\Omega|\nabla u|^2\,dx \quad\text{for every $u\in H_0^1(\Omega)$}.
 \end{equation}
 
 \item[iv)] When $\Omega = \R^N$, we have the characterization
 $$H^s(\R^N) = \big\{u\in L^2(\R^N):\,(1+|x|^{2s})|\mathcal{F}u|^2\in L^1(\R^N)\big\}. $$
\end{itemize}
The relation between the space $H^s(\Omega)$ and the fractional Laplacian $(-\Delta)^s$
is rooted in the following \emph{fractional integration-by-parts formula}: let 
$\Omega,\mathcal{V}\subseteq\R^N$ be open sets (bounded or not) such that $\Omega\Subset\mathcal{V}$,
and let $u\in C^2(\mathcal{V})\cap \mathcal{L}^s(\R^N)$; given any $\varphi\in C_0^\infty(\Omega)$, we have
\begin{equation} \label{eq:fractionalbyparts}
\begin{split}
 \int_{\Omega}(-\Delta)^su\,\varphi\,dx & = \frac{C_{N,s}}{2}\iint_{\R^{2N}}\frac{(u(x)-u(y))(\varphi(x)-\varphi(y))}{|x-y|^{N+2s}}\,dx\,dy \\
 & = \frac{C_{N,s}}{2}
 \iint_{\Omega\times\Omega}\frac{(u(x)-u(y))(\varphi(x)-\varphi(y))}{|x-y|^{N+2s}}\,dx\,dy
 \\
 &\qquad\quad-C_{N,s}\iint_{(\R^N\setminus\Omega)\times\Omega}
 \frac{(u(x)-u(y))\varphi(y)}{|x-y|^{N+2s}}\,dx\,dy
 \end{split}
\end{equation}
(note that, since $u\in C^2(\mathcal{V})\cap \mathcal{L}^s(\R^N)$, by
Proposition \ref{prop:welldefDeltas}
we have $(-\Delta)^su\in C(\mathcal{V})$).
As a con\-se\-quence of the above formula, it is then natural to define the
\emph{weak fractional Laplacian} $(-\Delta)^s u$
of a function $u\in H^s(\Omega)\cap \mathcal{L}^s(\R^N)$ as the
\emph{linear functional} acting on $C_0^\infty(\Omega)$ as follows
\begin{align*}
 (-\Delta)^s u(\varphi) & = \frac{C_{N,s}}{2}\iint_{\R^{2N}}\frac{(u(x)-u(y))
 (\varphi(x)-\varphi(y))}{|x-y|^{N+2s}}\,dx\,dy.
 \end{align*}
 Since we assuming that $u\in H^s(\Omega)\cap \mathcal{L}^s(\R^N)$, and since the kernel
 $|z|^{-n-2s}$ is integrable at infinity, it is easy to recognize 
 that the functional $(-\Delta)^s u$ is actually a \emph{distribution on $\Omega$}; more precisely,
 for every compact set $K\subseteq\Omega$ there exists a constant
 $C > 0$ such that
 \begin{equation} \label{eq:contDeltasu}
  \begin{gathered}
  \iint_{\R^{2N}}\frac{|u(x)-u(y)|
 |\varphi(x)-\varphi(y)|}{|x-y|^{N+2s}}\,dx\,dy 
 \leq C\,\mathrm{dist}(K,\R^N\setminus\Omega)^{-2s}\|\varphi\|_{H^s(\Omega)}
  \\[0.1cm]
  \text{for every $\varphi\in C_0^\infty(\Omega)$ with $\mathrm{supp}(\varphi)\subseteq K$}
  \end{gathered} 
 \end{equation}
 (here, the constant $C$ depends on $N,s$ and on $\|u\|_{1,s}$,
 see \eqref{eq:spaceL1s}), and this is enough to guarantee
 that $(-\Delta)^s u\in \mathcal{D}'(\Omega)$ (as $\|\varphi\|_{H^s(\Omega)}
 \leq c\,\|\varphi\|_{C^1(K)}$ for some absolute constant $c > 0$).
 \vspace*{0.05cm}  
 
 In the particular case when $u\in H^s(\R^N)\subseteq \mathcal{L}^s(\R^N)$,
  the above \eqref{eq:contDeltasu} shows that
  $(-\Delta)^s u$ can be continuously extended to 
  $H^s(\R^N)$
  (recall that $C_0^\infty(\R^N)$
  is \emph{dense} in $H^s(\R^N)$), and thus
  $$(-\Delta)^s u\in (H^s(\R^N))'.$$ 
  In this case, we also have
  $$\frac{\mathrm{d}}{\mathrm{d}t}\Big|_{t = 0}[u+tv]^2_{s,\R^N} = (-\Delta)^s u(v)\quad
  \text{for all $v\in H^s(\R^N)$}.$$
\medskip

\noindent\textbf{1.2) The weak maximum principle.} Thanks to the above preliminaries,
we can now establish a weak maximum principle for weak sub/supersolutions
of the equation $\LL u+V(x)u = 0$.
\vspace*{0.1cm}

\noindent\textbf{Notation.} Throughout what follows, if $\Omega\subseteq\R^N$ is a \emph{bounded}
open set and $u\in H_0^1(\Omega)$, we de\-no\-te by $\hat{u}$ the zero-extension of $u$ out of $\Omega$, 
that is,
$$\hat{u} := u\cdot\mathbf{1}_\Omega.$$
Since $u\in H^1_0(\Omega)$, we have $\hat{u}\in H^1(\R^N)$ and $\hat{u}\equiv 0$ a.e.\,in $\R^N\setminus\Omega$.
\medskip

To begin with, we give the following definitions.
\begin{definition} \label{def:weaksubsupersol}
 Let $\Omega\subseteq\R^N$ be a \emph{bounded} open set, and let $V\in L^2(\Omega)$. 
 Given any $f\in L^2(\Omega)$, we say that a function $u:\R^N\to\R$ is a \emph{weak
 subsolution [resp.\,\emph{supersolution}]} of
 \begin{equation} \label{eq:LLvzero}
  \LL u-V(x) = f\quad\text{in $\Omega$}
  \end{equation}
 if it satisfies the following properties:
 \begin{itemize}
  \item[a)] there exists a bounded open set $\Omega'$ such that $\Omega\Subset\Omega'$ and 
 $u\in H^1(\Omega')\cap \mathcal{L}^s(\R^N)$;
  \item[b)] for every test function $\varphi\in H^1_0(\Omega),\,\text{$\varphi\geq 0$ a.e.\,in $\Omega$}$ we have
  \begin{equation} \label{eq:weaksubsupersol}
  \begin{split}
   & -\int_\Omega\langle\nabla u,\nabla\varphi\rangle\,dx
   - \frac{C_{N,s}}{2}\iint_{\R^{2N}}\frac{(u(x)-u(y))(\hat{\varphi}(x)-\hat{\varphi}(y))}{|x-y|^{N+2s}}\,dx\,dy
   - \int_\Omega V(x)u\varphi\,dx \\
   & \qquad \geq\,[\text{resp.\,\,$\leq$}]\,
   \int_\Omega f(x)\varphi\,dx.
   \end{split}
  \end{equation}
 \end{itemize}
 We say that $u$ is a \emph{weak solution} of the equation
 \eqref{eq:LLvzero} if $u$ is both a weak subsolution and a
 weak supersolution of the same equation.
\end{definition}
\begin{remark} \label{rem:solwellposed}
 Let the assumptions and the notation of Definition \ref{def:weaksubsupersol} be in force.
 The regularity of $u$ on the \emph{larger} open set $\Omega'\Supset\Omega$ is somehow motivated
 by the \emph{nonlocal nature} of $\LL$ \emph{(}which is encoded in the nonlocal bilinear
 form $\langle\cdot,\cdot\rangle_{s,\,\R^N}$\emph{)}, and is essential to ensure that
 \begin{equation}  \label{eq:toprovedefwellposed}
  \iint_{\R^{2N}}\frac{|u(x)-u(y)||\hat{\varphi}(x)-\hat{\varphi}(y))|}{|x-y|^{N+2s}}\,dx\,dy<\infty\quad
 \text{for every $\varphi\in H_0^1(\Omega)$}.
 \end{equation}
 To see this we first observe that, 
 since $\Omega\Subset\Omega'$, it is possible to find a bounded open set $\mathcal{V}$ with
 \emph{smooth boundary} and such that $\Omega\Subset\mathcal{V}\Subset\Omega'$;
 thus, since $u\in H^1(\mathcal{V})$ \emph{(}recall that $u\in H^1(\Omega')$ and 
 that $\mathcal{V}\Subset\Omega'$\emph{)}, from the above \emph{iii)} we deduce that
 $u\in H^s(\mathcal{V})$. In view of this last fact, we are then entitled to
 apply \eqref{eq:contDeltasu} with the choice $K = \overline{\Omega}\subseteq\mathcal{V}$: this gives
 \begin{equation*}
 \begin{gathered}
  \iint_{\R^{2N}}\frac{|u(x)-u(y)||\varphi(x)-\varphi(y)|}{|x-y|^{N+2s}}\,dx\,dy
  \leq C\,\mathrm{dist}(\overline{\Omega},\R^N\setminus\mathcal{V})^{-2s}\|\varphi\|_{H^s(\mathcal{V})} \quad
  \forall\,\,\varphi\in C_0^\infty(\Omega).
 \end{gathered}
 \end{equation*}
 \emph{(}here, the constant $C > 0$ depends on $N,s$ and on $\|u\|_{1,s}$\emph{)}.
 Now, by combining this last displayed estimate with 
 the above \eqref{eq:H01embeddingHs} in \emph{ii)},
 we get
 \begin{equation} \label{eq:tousecontTg}
  \iint_{\R^{2N}}\frac{|u(x)-u(y)||\varphi(x)-\varphi(y)|}{|x-y|^{N+2s}}\,dx\,dy
  \leq C'\,\mathrm{dist}(\overline{\Omega},\R^N\setminus\mathcal{V})^{-2s}
  \Big(\int_\Omega|\nabla\varphi|^2\,dx\Big)^{1/2},
 \end{equation}
 for a suitable constant $C' > 0$, also depending on $\Omega$;
 from this, by taking into account that $C_0^\infty(\Omega)$ is dense in $H_0^1(\Omega)$, we
 easily
 conclude that \eqref{eq:toprovedefwellposed} holds.
 \vspace*{0.05cm}
 
 As a matter of fact, one can define a weak subsolution/supersolution of 
 equation \eqref{eq:LLvzero} as a function $u\in H^1_{\mathrm{loc}}(\Omega)$
 satisfying \eqref{eq:weaksubsupersol} for every non-negative test function
  $\varphi\in H_0^1(\mathcal{O})$,
 where $\mathcal{O}$ is as an arbitrary open set compactly contained in $\Omega$, see
 \cite{GaLi}; however, in order to e\-sta\-blish the weak maximum principle
 for \emph{non-negative} weak subsolutions of \eqref{eq:LLvzero} we will need to use
 $\varphi = u^-$ as a test function, 
 and such a function \emph{does not belong} to $H_0^1(\mathcal{O})$
 if $\mathcal{O}\Subset\Omega$
 \emph{(}rather, we will see that $u^-\in H_0^1(\Omega)$ as a consequence
 of the assumption $u\geq 0$ in $\R^N\setminus\Omega$\emph{)}.
 On the other hand, we will see in the next paragraph that this
 `improved' regularity out of $\Omega$ is satisfied by the \emph{weak solutions} of the $\LL$-Dirichlet
 problem, see Theorem \ref{thm:existenceDir} below.
\end{remark}
\begin{definition} \label{def:weaksolnonneg}
 Let $\Omega\subseteq\R^N$ be a \emph{bounded} open set, and let $u:\R^N\to\R$ be a measurable 
 fun\-ction. We say that
 $u\geq 0$ in $\R^N\setminus\Omega$ \emph{(}in the weak sense\emph{)} if
 $$\text{$u^- = \max\{-u,0\}\in H_0^1(\Omega)$ and $u\geq 0$ a.e.\,in $\R^N\setminus\Omega$}.$$
\end{definition}
We are now ready to prove the following weak maximum principle.
\begin{theorem} \label{thm:wmpsubsol}
 Let $\Omega\subseteq\R^N$ be a \emph{bounded} open set, and 
 let $V\in L^2(\Omega),\,\text{$V\geq 0$ a.e.\,in $\Omega$}$. 
 Moreover, let $u:\R^N\to\R$ be a weak supersolution of the equation $\LL u-V(x)u = 0$ in $\Omega$.
 
 If $u\geq 0$ in $\R^N\setminus\Omega$ \emph{(}in the weak sense\emph{)}, then $u\geq 0$ a.e.\,in $\Omega$.
\end{theorem}
The proof of Theorem \ref{thm:wmpsubsol} follows the same lines of that of
\cite[Thereon 1.2]{BDVV}, where the same result is proved in the case $V\equiv 0$ (and under
the stronger assumption $u\in H^1(\R^N)$); however, we present it here for the sake of completeness.
\begin{proof}
 First of all we observe that, since $u\geq 0$ in $\R^N\setminus\Omega$ (in the weak sense),
 we have $u^-\in H_0^1(\Omega)$ and $u^-\equiv 0$ a.e.\,in $\R^N\setminus\Omega$ (that is,
 $u^- = u^-\cdot\mathbf{1}_\Omega$); as a consequence, 
 since $u$ is a weak supersolution of the equation $\LL u - V(x)u = 0$ in $\Omega$
 (and since, obviously, $u^-\geq 0$ in $\Omega$), we are entitled
 to use the function $u^- = u^-\cdot\mathbf{1}_\Omega$ 
 as a test function in \eqref{eq:weaksubsupersol}, obtaining
 \begin{align*}
 & -\int_\Omega\langle\nabla u,\nabla u^-\rangle\,dx
   - \frac{C_{N,s}}{2}\iint_{\R^{2N}}\frac{(u(x)-u(y))(u^-(x)-u^-(y))}{|x-y|^{N+2s}}\,dx\,dy
   -\int_\Omega V(x)u\,u^-\,dx \leq 0.
 \end{align*}
 Now, since $V\geq 0$ a.e.\,in $\Omega$ and since $u\,u^- = -u^2\cdot\mathbf{1}_{\{u < 0\}}\leq 
 0$ a.e.\,in $\R^N$, we have
 $$\int_\Omega V(x)u\,u^-\,dx\leq 0;$$
 moreover, a direct computation shows that 
 $$\text{$(u(x)-u(y))(u^-(x)-u^-(y))\leq 0$ for a.e.\,$x,y\in\R^N$},$$ 
 and thus we have
 $$\frac{C_{N,s}}{2}\iint_{\R^{2N}}\frac{(u(x)-u(y))(u^-(x)-u^-(y))}{|x-y|^{N+2s}}\,dx\,dy\leq 0.$$
 Gathering all these facts, we then get
 \begin{equation}\label{eq:toconcludeWMP}
  \begin{split}
  &0 \geq -\int_{\Omega}\langle\nabla u,\nabla u^-\rangle\,dx = \int_{\Omega}|\nabla u^-|^2\,dx\geq 0.
 \end{split}
 \end{equation}
 Recalling that $u^-\in H_0^1(\Omega)$, from the above 
 \eqref{eq:toconcludeWMP} we infer that
 $u^- = \max\{-u,0\}\equiv 0$ a.e.\,in $\Omega$, and hence $u\geq 0$ a.e.\,in $\Omega$.
 This ends the proof.
\end{proof}
\subsection*{2) The $\LL_V$-Dirichlet problem.}
We conclude this appendix by briefly turning our attention
to the Dirichlet problem driven by the operator $\LL_V = \LL - V(x)$.
\vspace*{0.1cm}

To begin with, we give the following definition.
\begin{definition} \label{def:solDirPB}
 Let $\Omega\subseteq\R^N$ be a \emph{bounded} open set, and $V\in L^2(\Omega)$. Moreover,
 let $f\in L^2(\Omega)$ and let $g:\R^N\to\R$ be a measurable function.
 We say that a function $u:\R^N\to\R$ is a \emph{(}weak\emph{)} \emph{solution}
 of the $\LL_V$-Dirichlet problem
 $$\mathrm{(D)}\qquad\quad
 \begin{cases}
 \LL u -V(x)u = f & \text{in $\Omega$}, \\
 u = g & \text{in $\R^N\setminus\Omega$}
 \end{cases}$$
 if it satisfies the following properties:
 \begin{itemize}
  \item[a)] $u$ is a weak solution of the equation $\LL u -V(x)u = f$, 
  in the sense of Definition \ref{def:weaksubsupersol};
  \item[b)] $u-g\in H_0^1(\Omega)$ and $u\equiv g$ a.e.\,in $\R^N\setminus\Omega$.
 \end{itemize}
\end{definition}
\begin{remark} \label{rem:regulgDirPb}
 Let the assumptions and the notation of Definition \ref{def:solDirPB} be in force,
 and \emph{assume} that there exists a solution $u:\R^N\to\R$ of the $\LL_V$-Dirichlet problem $\mathrm{(D)}$.
 Then,
 $$\text{$g\in H^1(\Omega')\cap \mathcal{L}^s(\R^N)$ for some bounded open set $\Omega'\Supset\Omega$}.$$
 In fact, since $u$ is a weak solution of $\LL u -V(x)u = f$, by Definition
 \ref{def:weaksubsupersol} there exists a bounded open set $\Omega'\Supset\Omega$ such that
 $u\in H^1(\Omega')\cap \mathcal{L}^s(\R^N)$; thus, from property \emph{b)} we get
 $$g = (g-u)\cdot\mathbf{1}_\Omega+u\in H^1(\Omega')\cap \mathcal{L}^s(\R^N)$$
 \emph{(}recall that $(g-u)\cdot\mathbf{1}_\Omega\in H^1(\R^N)\subseteq \mathcal{L}^s(\R^N)$, as $u-g\in H_0^1(\Omega)$\emph{)}.
\end{remark}
We are now ready to prove the following existence result.
\begin{theorem} \label{thm:existenceDir}
 Let $\Omega\subseteq\R^N$ be a \emph{bounded} open set, and let 
 $V\in L^\infty(\Omega),\,\text{$V\geq 0$ a.e.\,in $\Omega$}$. Moreover,
 let $f\in L^2(\Omega)$ and let $g\in H^1(\Omega')\cap \mathcal{L}^s(\R^N)$, where $\Omega'\subseteq\R^N$
 is a bounded open set such that $\Omega\Subset\Omega'$.
 Then, there exists a \emph{unique weak solution} $u\in H^1(\Omega')\cap \mathcal{L}^s(\R^N)$
 of problem $\mathrm{(D)}$.
\end{theorem}
\begin{proof}[Proof (Existence).] On the real Hilbert space $H_0^1(\Omega)$, we consider the bilinear form
\begin{align*}
  \mathcal{B}_s(u,v) = \int_{\Omega}\langle \nabla u,\nabla v\rangle\,dx+\frac{C_{N,s}}{2}
  \iint_{\R^{2N}}\frac{(\hat{u}(x)-\hat{u}(y))(\hat{u}(x)-\hat{u}(y)}
  {|x-y|^{N+2s}}\,dx\,dy +\int_\Omega V(x)uv\,dx,
 \end{align*}
 and we claim that $\mathcal{B}_s$ is \emph{continuous and coercive on $H^1_0(\Omega)$}.
 \medskip
 
 \noindent -\,\,\emph{Continuity of $\mathcal{B}_s$.} First of all, by combining H\"older's inequality
 with \eqref{eq:H01embeddingHs}, we have
 \begin{align*}
  & \iint_{\R^{2N}}\frac{|\hat{u}(x)-\hat{u}(y)|\,|\hat{v}(x)-\hat{v}(y)|}{|x-y|^{N+2s}}\,dx\,dy \\
  & \qquad \leq 
  \Big(\iint_{\R^{2N}}\frac{|\hat{u}(x)-\hat{u}(y)|^2}{|x-y|^{N+2s}}\,dx\,dy\Big)^{1/2}
  \Big(\iint_{\R^{2N}}\frac{|\hat{v}(x)-\hat{v}(y)|^2}{|x-y|^{N+2s}}\,dx\,dy\Big)^{1/2} \\
  & \qquad \leq \gamma\,\|u\|_{H_0^1(\Omega)}\,\|v\|_{H_0^1(\Omega)},
 \end{align*}
 where $\gamma > 0$ is a suitable constant depending on $n,s$ and on $\Omega$, see \eqref{eq:H01embeddingHs}.
 In addition, since we are assuming $V\in L^\infty(\Omega)$, by H\"older's inequality
 and the Poincar\'e inequality (which can be applied without any further assumption, 
 since $u,v\in H_0^1(\Omega)$),
 we have
 \begin{align*}
 \int_\Omega|V(x)|\,|u|\,|v|\,dx & \leq \|V\|_{L^\infty(\Omega)}
 \int_\Omega|u|\,|v|\,dx \leq \|V\|_{L^\infty(\Omega)}\,\|u\|_{L^2(\Omega)}\,\|v\|_{L^2(\Omega)} \\
 & \leq \mathbf{c}_p(\Omega)\,\|V\|_{L^\infty(\Omega)}\,\|u\|_{H_0^1(\Omega)}\,\|v\|_{H_0^1(\Omega)},
 \end{align*}
 where $\mathbf{c}_p(\Omega) > 0$ is the Poincar\'{e} constant in $\Omega$.
 Gathering all these facts, we then infer that
 $$|\mathcal{B}_s(u,v)|\leq \big(1+\gamma+
 \mathbf{c}_p(\Omega)\|V\|_{L^\infty(\Omega)}\big)\|u\|_{H_0^1(\Omega)}
 \,\|v\|_{H_0^1(\Omega)}\equiv C\,\|u\|_{H_0^1(\Omega)}\,\|v\|_{H_0^1(\Omega)},$$
 and this proves that $\mathcal{B}_s$ is continuous on $H_0^1(\Omega)$.
 \vspace*{0.1cm}
 
 \noindent -\,\,\emph{Coercivity of $\mathcal{B}_s$.} By the very definition of $\mathcal{B}_s$, and since
 $V\geq 0$ a.e.\,in $\Omega$, we get
 \begin{align*}
  \mathcal{B}_s(u,u) & = \int_\Omega|\nabla u|^2\,dx
 + \frac{C_{N,s}}{2}\iint_{\R^{2N}}\frac{|\hat{u}(x)-\hat{u}(y)|^2}{|x-y|^{N+2s}}\,dx\,dy
 +\int_\Omega V(x)u^2\,d x\\
 & \geq \int_\Omega|\nabla u|^2\,dx = \|u\|_{H_0^1(\Omega)}^2,
 \end{align*}
 and this proves that $\mathcal{B}_s$ is coercive on $H_0^1(\Omega)$.
 \medskip
 
 \noindent We now consider the linear operator $T_g:H_0^1(\Omega)\to\R$ defined as follows
 \begin{align*}
  T_g(u) & := \int_\Omega \langle\nabla g,\nabla u\rangle\,dx
 - \frac{C_{N,s}}{2}
  \iint_{\R^{2N}}\frac{(g(x)-g(y))(\hat{u}(x)-\hat{u}(y))}{|x-y|^{N+2s}}\,dx\,dy \\
  & \qquad\qquad
  -\int_\Omega V(x)gu\,dx-
  \int_\Omega fu\,dx,
  \end{align*}
  and we observe that, since $g\in H^1(\Omega')\cap \mathcal{L}^s(\R^N)$ and since $f\in L^2(\Omega)$, 
  this operator
  $T_g$ is \emph{well-de\-fi\-ned and continuous on the space $H_0^1(\Omega)$} (see
  \eqref{eq:tousecontTg} in Remark \ref{rem:solwellposed}).
  
  In view of all these facts, we are in a position to apply the Lax-Milgram Theorem,
  ensuring the existence of a unique function $v_0\in H_0^1(\Omega)$ such that
  $$\mathcal{B}_s(v_0,\varphi) = T_g(\varphi)\quad\text{for every $\varphi\in H_0^1(\Omega)$}.$$
  Setting $u_0 := v_0+g$, it is then immediate to recognize that $u_0$
  is a weak so\-lu\-tion of the Dirichlet problem $\mathrm{(D)}$
  (in the sense of Definition \ref{def:solDirPB}). 
\medskip

\noindent \emph{Proof (Uniqueness).} Let $u_1,u_2$ be two weak solutions of
the Dirichlet problem $\mathrm{(D)}$
(in the sense of Definition \ref{def:weaksubsupersol}). Setting $v := u_1-u_2$, it is readily
seen that $v$ is a weak supersolution of the equation $\LL u-V(x)u = 0$ in $\Omega$, and $v\geq 0$
in $\R^N\setminus\Omega$ (in the weak sense); then, by the Weak Maximum Principle
in Theorem \ref{thm:wmpsubsol} we get $v\geq 0$ a.e.\,in $\Omega$, that is,
$$\text{$u_1\geq u_2$ a.e.\,in $\Omega$}.$$
On the other hand, by interchanging the roles of $u_1$ and $u_2$ we also get $u_2\geq u_1$ a.e.\,in $\Omega$.
As a consequence, since $u_1 = u_2 = g$ a.e.\,in $\R^N\setminus\Omega$ (as $u_1$ and $u_2$
are solutions of the same problem), we conclude that $u_1 = u_2$ a.e.\,in $\R^N$,
and the proof is complete.
\end{proof}
\begin{remark} \label{rem:assumptionVal}
 Let $\Omega\subseteq\R^N$ be a \emph{bounded} open set, and let 
 $V\in L^\infty(\Omega),\,\text{$V\geq 0$ a.e.\,in $\Omega$}$. Moreo\-ver,
 let $f\in L^2(\Omega)$ and let $\alpha\in\R$. Denoting by $u_\alpha$
 the unique solution of
 $$\mathrm{(D)_\alpha}\qquad\quad\begin{cases}
 \LL u -V(x)u = f & \text{in $\Omega$}, \\
 u = \alpha & \text{in $\R^N\setminus\Omega$}
 \end{cases}$$
 \emph{(}whose existence and uniqueness is guaranteed by Theorem \ref{thm:existenceDir}\emph{)}, we have
 $$\iint_{\R^{2N}}\frac{|u(x)-u(y)|^2}{|x-y|^{N+2s}}\,dx\,dy < \infty.$$
 In fact, since $u$ is a solution of $\mathrm{(D)_\alpha}$, by Definition \ref{def:solDirPB} we have
 $$\text{$u = (u-\alpha)\cdot\mathbf{1}_\Omega+\alpha$ a.e.\,in $\R^N$};$$
 thus, since $v:=(u-\alpha)\cdot\mathbf{1}_\Omega\in H^1(\R^N)$ 
 \emph{(}as $u-\alpha\in H_0^1(\Omega)$\emph{)},
 from \eqref{eq:H01embeddingHs} we obtain
 \begin{align*}
  \iint_{\R^{2N}}\frac{|u(x)-u(y)|^2}{|x-y|^{N+2s}}\,dx\,dy
  & = \iint_{\R^{2N}}\frac{|v(x)-v(y)|^2}{|x-y|^{N+2s}}\,dx\,dy
  \leq \gamma\,\|v\|_{H^1_0(\Omega)} = \gamma\,\|u-\alpha\|_{H^1_0(\Omega)},
 \end{align*}
 where $\gamma > 0$ is a suitable constant depending on $n,s$ and on $\Omega$.
 We also notice that, since the constant function $g = \alpha\in H^1(\Omega')\cap \mathcal{L}^s(\R^N)$
 \emph{for every bounded open set $\Omega'\Supset\Omega$}, we get
 $$u_\alpha\in H^1_{\mathrm{loc}}(\R^N)\cap \mathcal{L}^s(\R^N).$$
\end{remark}

\bigskip

\noindent{\bf{Acknowledgments.}} The second author is funded by the Deutsche Forschungsgemeinschaft (DFG, German Research Foundation) - SFB 1283/2 2021 - 317210226. All authors are member of the ``Gruppo Nazionale per l'Analisi Matematica, la Probabilit\'a e le loro Applicazioni'' (GNAMPA) of the ``Istituto Nazionale di Alta Matematica'' (INdAM, Italy).

\end{document}